\documentclass[titlepage]{amsart}
\usepackage{amsaddr}

\newtheorem{pro}{Proposition}[section]
\newtheorem{teo}[pro]{Theorem}
\newtheorem{defi}[pro]{Definition}
\newtheorem{lem}[pro]{Lemma}
\newtheorem{cor}[pro]{Corollary}

\newtheorem{ex}[pro]{Example}

\newtheorem{prop}[pro]{Proposition}
\newtheorem{thm}[pro]{Theorem}
\newtheorem{defn}[pro]{Definition}

\newtheorem{example}[pro]{Example}

\newcommand{\Ext}{\mathrm{Ext}}
\newcommand{\Hom}{\mathrm{Hom}}

\newcommand{\A}{\mathcal{A}}
\newcommand{\B}{\mathcal{B}}

\newcommand{\C}{\mathcal{C}}

\newcommand{\T}{\mathcal{T}}
\newcommand{\U}{\mathcal{U}}
\newcommand{\V}{\mathcal{V}}
\newcommand{\X}{\mathcal{X}}
\newcommand{\Y}{\mathcal{Y}}
\newcommand{\Z}{\mathcal{Z}}

\newcommand{\pd}{\mathrm{pd}}

\newcommand{\id}{\mathrm{id}}

\newcommand{\resdim}{\mathrm{resdim}}

\newcommand{\coresdim}{\mathrm{coresdim}}
\newcommand{\add}{\mathrm{add}}
\newcommand{\smd}{\mathrm{smd}}
\newcommand{\Add}{\mathrm{Add}}

\newcommand{\modu}{\mathrm{mod}}

\newcommand{\Ker}{\mathrm{Ker}}
\newcommand{\Ima}{\mathrm{Im}}

\usepackage{latexsym,amssymb,amscd} 
\usepackage{amsmath}
\usepackage[all]{xypic}
\usepackage{amsthm}
\usepackage{xargs}[2008/03/08] 

\newcommandx\suc[5][usedefault, addprefix=\global, 1=N, 2=M, 3=K, 4=, 5=]{0\rightarrow#1\overset{#4}{\rightarrow}#2\overset{#5}{\rightarrow}#3\rightarrow0}
\newcommandx\p[2][usedefault, addprefix=\global, 1=\mathcal{A}, 2=\mathcal{B}]{\left(#1,#2\right)}
\newcommandx\pdr[2][usedefault, addprefix=\global, 1=\mathcal{A}, 2=M]{\mathrm{pd}{}_{#1}\left(#2\right)}
\newcommandx\idr[2][usedefault, addprefix=\global, 1=\mathcal{B}, 2=M]{\mathrm{id}{}_{#1}\left(#2\right)}
\newcommandx\Extx[4][usedefault, addprefix=\global, 1=i, 2=\mathcal{C}, 3=M, 4=X]{\mathrm{Ext}{}_{#2}^{#1}\left(#3,#4\right)}
\newcommandx\Injx[1][usedefault, addprefix=\global, 1=R]{\operatorname{Inj}\left(#1\right)}
\newcommandx\Projx[1][usedefault, addprefix=\global, 1=R]{\operatorname{Proj}\left(#1\right)}
\newcommandx\smdx[1][usedefault, addprefix=\global, 1=\mathcal{M}]{\operatorname{smd}\left(#1\right)}
\newcommandx\Kerx[1][usedefault, addprefix=\global, 1=M]{\operatorname{Ker}\left(#1\right)}
\newcommandx\Homx[3][usedefault, addprefix=\global, 1=\mathcal{C}, 2=M, 3=N]{\mathrm{Hom}{}_{#1}(#2,#3)}
\newcommandx\End[2][usedefault, addprefix=\global, 1=R, 2=M]{\mathrm{End}{}_{#1}(#2)}
\newcommandx\pdx[1][usedefault, addprefix=\global, 1=M]{\operatorname{pd}\left(#1\right)}
\newcommandx\Cok[1][usedefault, addprefix=\global, 1=M]{\operatorname{Coker}\left(#1\right)}
\newcommandx\Gen[1][usedefault, addprefix=\global, 1=M]{\operatorname{Fac}_{1}\left(#1\right)}
\newcommandx\Genn[2][usedefault, addprefix=\global, 1=M, 2=n]{\operatorname{Fac}_{#2}(#1)}
\newcommandx\Gennr[3][usedefault, addprefix=\global, 1=\mathcal{T}, 2=n, 3=\mathcal{X}]{\operatorname{Fac}_{#2}^{#3}(#1)}
\newcommandx\addx[1][usedefault, addprefix=\global, 1=\mathcal{M}]{\operatorname{add}\left(#1\right)}
\newcommandx\Modx[1][usedefault, addprefix=\global, 1=R]{\operatorname{Mod}\left(#1\right)}
\newcommandx\modd[1][usedefault, addprefix=\global, 1=R]{\operatorname{mod}\left(#1\right)}
\global\long\def\im#1{\operatorname{Im}\left(#1\right)}
\global\long\def\resdimr#1#2#3{\mathrm{resdim}{}_{#1}^{#3}\left(#2\right)}
\global\long\def\coresdimr#1#2#3{\mathrm{coresdim}{}_{#1}^{#3}\left(#2\right)}
\global\long\def\resdimx#1#2{\operatorname{resdim}_{#1}\left(#2\right)}
\global\long\def\coresdimx#1#2{\operatorname{coresdim}_{#1}\left(#2\right)}

\usepackage{tikz}
\usepackage{pgfplots}
\usetikzlibrary{matrix,arrows,positioning,calc}
\usepackage[all]{xy}
\usepackage{amssymb,amsmath}
\usetikzlibrary{positioning}
\usetikzlibrary{backgrounds}
\usepackage{enumitem}
\setenumerate{ label= (\alph*)}
\setenumerate[2]{ label= (\alph{enumi}\arabic*)} 

\usepackage[Lenny]{fncychap}
\usepackage{amscd}

\usepackage{etoolbox}

\makeatletter
\patchcmd{\@maketitle}
{\if@titlepage \newpage \else}
{\if@titlepage
 \vspace{\baselineskip}
 \else}
{}{}
\makeatother

\begin{document}

\title[Relative Tilting Theory I]{Relative tilting theory in abelian categories I: Auslander-Buchweitz-Reiten approximations theory in subcategories and cotorsion-like pairs}

\author{Alejandro Argud\'in-Monroy}
\address{Centro de Ciencias Matem\'aticas, Campus Morelia, Universidad Nacional
Autónoma de M\'exico, Antigua Carretera a Pátzcuaro 8701, Colonia Ex Hacienda
San Jos\'e de la Huerta, Morelia, 58089 Michoac\'an, MEXICO.}

\author{Octavio Mendoza Hern\'andez}
\address{Instituto de Matem\'aticas, Universidad Nacional Aut\'onoma de M\'exico.  
Circuito Exterior, Ciudad Universitaria, 
C.P. 04510, CDMX,  MEXICO.}

\email[A1,A2]{argudin@ciencias.unam.mx, omendoza@matem.unam.mx}

\thanks{
\indent {\em Funding: } This work was supported by the Projects PAPIIT-Universidad Nacional Aut\'onoma de M\'exico IN100520 and IN100124. The first named author was supported by a postdoctoral fellowship from Programa
de Becas Posdoctorales en la UNAM (POSDOC), Direcci\'on General de Asuntos del Personal Acad\'emico,
Universidad Nacional Aut\'onoma de M\'exico.
 }

\begin{abstract} 
In this paper we introduce a special kind of relative (co)resolutions 
associated to a pair of classes of objects in an abelian category $\C.$ We will see
that, by studying  this relative (co)resolutions, we get a possible generalization of a part of the Auslander-Buchweitz approximation theory that is useful for developing $n$-$\X$-tilting theory in \cite{Argudin-Mendoza2}. With this goal, new
concepts as $\mathcal{X}$-complete and $\mathcal{X}$-hereditary
pairs are introduced as a generalization of complete and hereditary
cotorsion pairs. These pairs appear in a natural way in the study of the category of  representations of a quiver in an abelian category \cite{Argudin-Mendoza3}.
Our main results will include an existence theorem  for relative approximations, among other results related with closure properties of relative (co)resolution classes and relative homological dimensions which are essential in the development of $n$-$\X$-tilting theory in \cite{Argudin-Mendoza2}.
\end{abstract}  
\maketitle

\setcounter{tocdepth}{1}
\tableofcontents

\section{Introduction}

Broadly speaking, tilting modules over a ring were born with the purpose of answering
questions, about the ring we are working on, through the endomorphisms ring
of such tilting module. With this idea, tilting modules have become an important
tool in different areas like Representation theory,  homological algebra and theory of categories \cite{hugel2007handbook}.

In the last 40 years, tilting theory has been generalized in different
ways and contexts. Namely, it has been generalized from finitely generated
\cite{happel1982tilted} to infinitely generated modules \cite{Tiltinginfinitamentegenerado};
from finite \cite{miyashita} to infinite projective dimension \cite{positselski2019tilting};
from categories of modules to abstract categories \cite{positselskicorrespondence}.
All these generalizations give rise to a family of different tilting definitions with different
properties and objectives. 
\

In this long history of exploring and expanding tilting theory, we
would like to mention the contribution made by M. Auslander and
I. Reiten in \cite{auslandereiten, auslander1992homologically}.
They were interested in the study of covariantly and contravariantly
finite subcategories in the category $\modu(\Lambda)$ of finitely generated left $\Lambda$-modules, where $\Lambda$ is an
Artin algebra. One of the main results of such papers is that they
found a bijection between tilting $\Lambda$-modules and covariantly
finite categories: \emph{the Auslander-Reiten correspondence} \cite[Thm. 4.4]{auslander1992homologically}.
By using this bijection,  M. Auslander and I. Reiten showed the utility of the
tilting objects for studying covariantly finite subcategories and
vice versa. 
\

This manuscript is  the first of two  forthcoming papers and is devoted to develop certain foundational
aspects needed for the settle and the study of a relative tilting theory on abelian
categories following the same philosophy of M. Auslander and I. Reiten
in \cite{auslandereiten,auslander1992homologically}.
We shall be based mainly 
in
 the first part of the homological algebra presented by M.
Auslander and R.O. Buchweitz in \cite{Auslander-Buchweitz}.
We can describe such theory as the study of the (co)resolution
dimension over a class $\mathcal{X},$  the relative
projective (injective) dimension on $\mathcal{X},$ some closure properties of certain classes and the existence of $\X$-precovers or $\X$-preenvelopes.  In this paper,
we will introduce a subtle modification on the studied (co)resolutions
which will lead us to new definitions and results
on relative homological algebra that are useful for the development of $n$-$\X$-tilting theory in \cite{Argudin-Mendoza2} and the study of representations of quivers in abelian categories \cite{Argudin-Mendoza3}.
\

 In the second forcoming paper  \cite{Argudin-Mendoza2}, by using the fundaments developed in this first work, we set and study a  theory
of relative tilting classes in abelian categories. The idea of this theory is that, if we have a class $\X$ in an abelian category $\C$, then we can use the structure of $\C$ to define a tilting subcategory $\T$ in $\C$ which is related with $\X,$ that is the so-called $n$-$\X$-tilting subcategory. Notice that, in general, $\T$ has not to be contained in $\X$ and thus this leads us to study cotorsion-like pairs introduced in section 3 of this first paper.  We will show, in the second paper, 
that this work offers a unified framework of different previous notions
of tilting, ranging from Auslander-Solberg relative tilting modules
on Artin algebras to infinitely generated tilting modules on arbitrary
rings.  With this new approach, we will review Bazzoni's
tilting characterization, relative homological dimensions on the induced
tilting classes, and parametrise certain cotorsion-like pairs. As an example, we will show how the tilting theory
in exact categories built this way, coincides with tilting objects
in extriangulated categories introduced recently by 
B. Zhu
 and
  X. Zhuang
    \cite{zhu2019tilting}. It is worth mentioning that a relative tilting theory  was recently  presented by 
  P. Moradifar
     and
 S. Yassemi
       in \cite{moradifar2021infinitely} with the goal of  studying infinitely generated Gorenstein tilting objects. We believe that our work will be a complementary tool for this research line.
\

Let us describe briefly  the contents and main results of this first paper. Section
2 is devoted to introduce some categorical and homological preliminaries
in an arbitrary abelian category $\C$. In particular, we
shall recall the definition of homological dimensions relative to
a class $\X\subseteq\C$. Namely, for a class $\T\subseteq\C,$ $\pd_\X(\T)$ denotes the $\mathcal{X}$-projective dimension of $\T$ and $\id_\X(\T)$ denotes the $\mathcal{X}$-injective dimension of $\T.$ Furthermore, in
Proposition \ref{prop:M ortodonal cerrado por n-cocientes sii pdM=00003Dn}  we  characterize the classes $\mathcal{T}\subseteq\mathcal{C}$
such that $\pd_\X(\T)\leq n$ through a property
called \emph{closed by $n$-quotients in $\mathcal{X}$}. We also present the definition of $n$-$\X$-cluster tilting in the abelian category $\C,$ which is a generalization of the $n$-cluster tilting category given by O. Iyama in \cite{iyama2011cluster}, and study such categories from the point of view of $n$-quotients in 
$\mathcal{X}\subseteq\C.$
\

In Section 3, we present and study new notions that will help us to state
a relative Auslander-Reiten correspondence in the second forcoming paper, see 
 \cite[Cors. 3.49 and 3.50]{Argudin-Mendoza2}. 
Namely, we shall introduce
the notions of \emph{$\X$-complete} and \emph{$\X$-hereditary} pairs. These concepts are generalizations of the complete 
and hereditary cotorsion pairs with the difference that the above-mentioned
pairs do not need to be cotorsion pairs. We point out that $\X$-complete and $\X$-hereditary pairs appear in a natural way in \cite{Argudin-Mendoza2} and are of the form $({}^\perp(\T^\perp), \T^\perp)$ for $\T$ an $n$-$\X$-tilting subcategory of $\C.$ Notice that we do not know if $({}^\perp(\T^\perp), \T^\perp)$ is a cotorsion pair and a particular case when this occur can be seen in Corollary \ref{Tp=cp}. Other context where the $\X$-complete and $\X$-hereditary pairs appear in a natural way is in the study of quiver representations in abelian categories, see 
\cite[Thms. 5.16 and 5.17 and Cor. 5.18]{Argudin-Mendoza3}.
\

Section 4 is devoted to develop our main goals. We introduce the type of relative (co)resolutions that interests us
and that will be useful for the development of the theory related with some special
classes, their closure properties and the relationship between different relative homological dimensions. Namely, let
 $\X,\Y\subseteq\C$ be classes of objects in an abelian category $\mathcal{C}.$ We introduce, see Definition \ref{def: res cores},
$\Y_\X$-(co)resolutions and the respective (co)resolution dimension associated
to the pair $(\mathcal{X},\mathcal{Y})$. These relative (co)resolutions give rise to the relative coresolution classes 
$(\X,\Y)^\vee_{\infty}$ and $(\X,\Y)^\vee,$ and to the relative resolution classes $(\X,\Y)^\wedge_{\infty}$ and $(\X,\Y)^\wedge.$ We point out that these classes already appeared as particular cases in classical tilting theory and Gorenstein homological algebra. Of course, in these particular situations, some closure properties has been studied. In the general case, we studied in this paper their closure properties in 
Theorem \ref{thm:teo nuevo}, Theorem \ref{thm:el reemplazo } and Corollary \ref{cor:coronuevo}. These closure properties  of the aforementioned  relative classes play an important role in the study and development of $n$-$\X$-tilting theory in the second forthcoming paper \cite{Argudin-Mendoza2}. On the other hand, Theorem \ref{thm:(5.4homologia_relativa)-1} shows the existence of the main approximations that we will be using in the case of relative coresolution classes. This theorem is a possible generalization of the dual
result of \cite[Thm 1.1]{Auslander-Buchweitz} and will play an important role in the development of
$n$-$\X$-tilting theory in \cite{Argudin-Mendoza2}. For example, it is fundamental in the proof that the pair
$({}^\perp(\T^\perp), \T^\perp)$ is $\X$-complete if $\T$ is an $n$-$\X$-tilting class in $\C.$ Finally, since the aforementioned pair is $\X$-complete and $\X$-hereditary, it has sense to study the relationship between the different relative homological dimensions related to any $\X$-complete and $\X$-hereditary pair $(\A,\B)$ in $\C.$ The main results obtained in this direction are Proposition \ref{pro:(5.7homologia_relativa)} and Theorem \ref{thm:(5.8homologia_relativa)} that will play an important role in the development
of the $n$-$\X$-tilting theory in \cite{Argudin-Mendoza2}.
\

Let $\mathcal{X}$ and $\mathcal{T}$ be classes of objects in an abelian category $\C.$ In Section
5, we  study a new class of objects $C\in\mathcal{C}$ admitting
an exact sequence 
$$0\to K\to T_n\xrightarrow{f_n} T_{n-1}\to\cdots\to T_2\xrightarrow{f_2} T_1\xrightarrow{f_1}C\to 0,$$
with $T_{i}\in\mathcal{T}\cap\mathcal{X}$ and $\Ker(f_{i})\in\X$
$\forall i\in[1,n]$. This is a generalization of the class
of modules $n$-generated by a given module, which were introduced
by 
S. Bazzoni
 \cite{Bazzonintilting} and 
 J. Wei
  \cite{Wei}
as a tool in the characterization of tilting modules. The goal of this section is to review some basic properties of such class that will be
used in \cite{Argudin-Mendoza2} to characterize when a class $\T$ in $\C$ is  $n$-$\X$-tilting.
\

It is worth 
pointing
 out that we will be working in abstract abelian
categories without assuming the existence of enough projectives or
injectives. Furthermore, we will incorporate examples where our theory
can be exploited. Namely, we shall include examples on $n$-cluster
tilting categories, $\mathcal{FP}_{n}$ objects and relative Gorenstein
objects among others.
\section{Preliminaries }

\subsection{Notation }

Throughout the paper,  $\mathcal{C}$ denotes an abelian category. We will
use the Grothendieck's notation \cite{Ab} to distinguish abelian
categories with further structure. Namely, \textbf{AB3 } (if it has coproducts), \textbf{AB4 } (if it is AB3 and the coproduct functor
is exact) and \textbf{AB5 } (if it is AB3 and the direct limit functor is exact).

Given $X,Y\in\mathcal{C}$ and $n\geq0,$ we will consider the $n$-th
Yoneda extensions group $\Ext^n_{\C}(X,Y)$ \cite[Chap. VII]{mitchell}
 and the bi-functor $\Ext^n_{\C}(-,-):\C^{op}\times\mathcal{C}\to\mbox{Ab}.$ 
In particular, if $\mathcal{C}$ has enough projectives or injectives,
this bi-functor coincides with the $n$-derived functor of the Hom functor.
For any class of objects $\mathcal{X}\subseteq\mathcal{C}$ and any $i\geq 1,$ we define  the right $i$-th orthogonal class
$\X^{\perp_i}:=\{C\in\C\;:\;\Ext^i_{\C}(-,C)|_{\X}=0\}$ and the right orthogonal class $\X^{\perp}:=\cap_{i>0}\;\X^{\perp_i}$ 
of $\X.$ Dually, we have the left $i$-th orthogonal class ${}^{\perp_i}\X$ and the left orthogonal class ${}^{\perp}\X$ of $\X.$ 
\

With respect to inclusions of classes and objects, $\mathcal{M}\subseteq\mathcal{C}$  means
that $\mathcal{M}$ is a class of objects of $\mathcal{C},$ and $C\in\mathcal{C}$ means that $C$ is an object
of $\mathcal{C}$. In the case we are given another class of objects
$\mathcal{N\subseteq\mathcal{C}}$, then $\mathcal{N}\subseteq\mathcal{M}$
and $C\in\mathcal{M}$ have similar meanings. Finally, the term subcategory means full subcategory.
\

Associated with some $\mathcal{M}\subseteq\C,$ we have the following classes of objects in  $\C.$ Namely: the class 
$\smd(\mathcal{M})$ whose objects are all the direct summands of objects in $\mathcal{M};$ the class $\mathcal{M}^{\oplus}$ 
($\mathcal{M}^{\oplus_{<\infty}}$) of all the  (finite) coproducts of objects in $\mathcal{M};$ $\Add(\mathcal{M}):=\smd(\mathcal{M}^{\oplus})$ and $\add(\mathcal{M}):=\smd(\mathcal{M}^{\oplus_{<\infty}}).$ In case $\mathcal{M}=\{M\},$ 
we have $M^{\oplus}:=\mathcal{M}^{\oplus},$
$M^{\oplus_{<\infty}}:=\mathcal{M}^{\oplus_{<\infty}},$ $\smd(M):=\smd(\mathcal{M}),$
$\Add(M):=\Add(\mathcal{M}),$ $\add(M):=\add(\mathcal{M}),$ $M^{\bot}:=\mathcal{M^{\bot}},$
and $^{\bot}M:={}^{\bot}\mathcal{M}.$
\

Throughout this paper, we will work with a variety of concepts along
with its dual notions. We will omit writing down dual results and notions, but
we will be using both of them. 

\subsection{Relative homological dimensions}

We will be using the following known results  in homological algebra and its duals. 

\begin{teo} \cite[Chap. VI, Thm. 5.1]{mitchell} Let $0\to N\to M\to K\to 0$ be an exact
sequence in the abelian category $\C$. Then, for any $X\in\C,$  there is a long exact sequence of abelian groups induced
by $\Hom_\C(X,-)$
\begin{align*}
0 & \rightarrow\Hom_\C(X,N)\rightarrow\Hom_\C(X,M)\rightarrow\Hom_\C(X,K)\rightarrow\Ext^1_\C(X,N)\rightarrow\cdots\\
\cdots & \rightarrow\Ext^k_\C(X,N)\rightarrow\Ext^k_\C(X,M)\rightarrow\Ext^k_\C(X,K)\rightarrow\Ext^{k+1}_\C(X,N)\rightarrow\cdots\mbox{;}
\end{align*}
\end{teo}

\begin{lem} [Shifting Lemma]\label{lem:recorrimiento de dimension} Let $0\to K  \to C_{n-1}\to\cdots\to C_{1}\to C_{0}\to A\to 0$ be an exact sequence in the abelian category $\C$ such that  $C_{i}\in{}^{\bot}Y$ $\forall i\in[0,n-1],$ where $Y\in\C.$ Then, 
$\Ext^k_\C(K,Y)\cong\Ext^{k+n}_\C(A,Y)\,\forall k\geq 1.$
\end{lem}
\begin{proof}
It can be proved in a similar way as in \cite[Chap. VII, Lem. 6.3]{mitchell}.
\end{proof}

\begin{teo} \label{prop:extn coprods arb}\cite[Prop. 4.2]{argudin2019yoneda}
Let $\mathcal{C}$ be an AB4 category, $\{A_{i}\}_{i\in I}$
be a family of objects in $\C$ and $B\in\mathcal{C}.$ Then, for any $n\geq1,$
there is a natural isomorphism 
\[
\Psi_{n}:\Ext^n_\C(\bigoplus_{i\in I}\, A_i, B)\rightarrow\prod_{i\in I}\Ext^n_\C(A_i,B).
\]
\end{teo}

For $\B,\A\subseteq\C$
and $C\in\C,$  we recall the following notions \cite{Auslander-Buchweitz}.  The \textbf{$\mathcal{A}$-projective}  dimension of $C$ is $\pd_\A(C):=\min\left\{ n\in\mathbb{N}\;:\;
\Ext^k_\C(C,-)|_\A=0\,\forall k>n\right\};$
the $\mathcal{A}$-projective dimension  of $\mathcal{B}$
is $\pd_\A(\B):=\sup\left\{ \pd_\A(B)\;:\; B\in\B\right\}.$ The \textbf{$\mathcal{A}$-injective}  dimension $\id_\A(C)$ of $C$ and  $\mathcal{A}$-injective dimension $\id_\A(\B)$ of $\mathcal{B}$
are defined dually. For a pair $(\mathcal{X},\omega)\subseteq\mathcal{C}^2,$ it is said that $\omega$ is a \textbf{relative cogenerator}  in $\mathcal{X}$ if $\omega\subseteq\mathcal{X}$ and any $X\in\mathcal{X}$ admits an
exact sequence $0\to X\to W\to X'\to 0$ in $\C,$ with $W\in\omega$ and $X'\in\mathcal{X}.$ The class $\omega$ is \textbf{$\mathcal{X}$-injective} if $\id_\X(\omega)=0.$ Dually, we have the notions of \textbf{relative generator}  in $\mathcal{X}$ and 
\textbf{$\mathcal{X}$-projective} class.

\subsection{Relative $n$-quotients and $n$-subobjects}

\begin{defi} Let $\C$ be an abelian category, $\Y\subseteq\X\subseteq\C$ and $n\geq1.$ 
\begin{itemize}
\item[(a)] $\Y$ is \textbf{closed by $n$-quotients in $\mathcal{X}$}
if for any exact sequence in $\C$
\[0\to A\to Y_{n}\xrightarrow{\varphi_{n}} Y_{n-1}\to\cdots\to Y_{1}\xrightarrow{\varphi_{1}}B\to 0, \]
with $Y_{i}\in\Y,$ $\Ker(\varphi_{i})\in\X$ $\forall i\in[1,n]$
and $B\in\X,$ we have that $B\in\Y.$
\item[(b)] $\Y$ is \textbf{closed by $n$-subobjects in $\mathcal{X}$}
if for any exact sequence  in $\C$
\[0\to A\to Y_{1}\xrightarrow{\varphi_{1}}Y_2\to\cdots\to Y_{n}\xrightarrow{\varphi_{n}}B\to 0,\]
with $Y_{i}\in\Y,$ $\Ima(\varphi_{i})\in\X$ $\forall i\in[1,n]$
and $A\in\X$, we have that $A\in\Y.$
\end{itemize}
\end{defi}

\begin{ex} \label{exa:fpn} Let $n$ be a positive  integer, and $\mathcal{C}$ be  an AB5 category admitting a set $\omega$
such that $\omega^{\oplus}$ is a relative generator in $\mathcal{C}.$
\begin{enumerate}
\item[(1)]   Denote by
$\mathcal{FP}_{n}$  the class of all the objects $X\in\mathcal{C}$
such that the functor $\Ext^i_\C(X,-)$ preserves direct limits $\forall i\in[0,n-1].$
By \cite[Lem. 2.11]{bravo2019locally}, $\mathcal{FP}_{n}$ is closed
by $(n+1)$-quotients in $\mathcal{C}$ if $\omega\subseteq\mathcal{F}\mathcal{P}_{1}.$
 
\item[(2)]  Denote by
 $\mathcal{C}_{n}$ the class of all the objects $X\in\mathcal{FP}_{n}$
such that every subobject $Y\subseteq X$ in the class $\mathcal{FP}_{n-1}$
is in fact in $\mathcal{FP}_{n}.$ Then, by \cite[Cor. 4.5]{bravo2019locally},
$\mathcal{C}_{n}$ is closed by $1$-quotients in $\mathcal{FP}_{n-1}.$
\end{enumerate}
\end{ex}

 The following connection between ``closed under $n$-quotients'' and the relative projective dimension will be very useful in our paper 
 \cite{Argudin-Mendoza2}, where we develop the theory of $n$-$\X$-tilting classes.
 
\begin{pro} \label{prop:M ortodonal cerrado por n-cocientes sii pdM=00003Dn}
Let $\mathcal{X},\mathcal{T}\subseteq\mathcal{C}$
 and $\alpha\subseteq\mathcal{T}^{\bot}\cap\X^\perp$ be a relative cogenerator
in $\mathcal{X}.$
Then, for $n\geq1,$ $\mathcal{X}\cap\mathcal{T}^{\bot}$ is closed by $n$-quotients
in $\mathcal{X}$ if, and only if, $\pd_\X(\T)\leq n.$
\end{pro}
\begin{proof}
$(\Rightarrow)$ Note that $\alpha\subseteq\mathcal{X}\cap\mathcal{T}^{\bot}.$
Let $X\in\mathcal{X}$ and $M\in\mathcal{T}.$ Since $\alpha$ is
a relative cogenerator in $\mathcal{X},$ there is an exact sequence $0\rightarrow X\rightarrow I_{0}\rightarrow\cdots\to 
 I_{n-1}\xrightarrow{f}V\rightarrow 0,$ 
 with $K:=\Ker(f)\in\X,$ $V\in\X$ and $I_{i}\in\alpha$
$\forall i\in[0,n-1]$. By applying the functor $\Hom_\C(M,-)$ to the exact sequence $0\to K\to I_{n-1}\to V\to 0$ and since 
$\Ext^i_\C(M,I_{n-1})=0=\Ext^{i+1}_\C(M,I_{n-1}),$ it follows that $\Ext^i_\C(M,V)\cong\Ext^{i+1}_\C(M,K)$ 
$\forall\;i\geq 1.$ By the dual of Lemma \ref{lem:recorrimiento de dimension}, 
$\Ext^{i+1}_\C(M,K)\cong\Ext^{n+i}_\C(M,X).$
Now, since $\mathcal{T}^{\bot}\cap\mathcal{X}$ is closed by $n$-quotients
in $\mathcal{X}$ and $V\in\mathcal{T}^{\bot}\cap\mathcal{X},$ we get
\[0=\Ext^i_\C(M,V)\cong\Ext^{i+1}_\C(M,K)\cong\Ext^{n+i}_\C(M,X)\;\forall i\geq1;\]
and thus $\pd_\X(M)\leq n.$
\

$(\Leftarrow)$ Let $0\rightarrow A\rightarrow X_{n}\xrightarrow{\varphi_n}X_{n-1}\to\cdots\to X_{1}\xrightarrow{\varphi_1} B\rightarrow 0$
be an exact sequence, where $X_{1},\cdots,X_{n}\in\mathcal{X}\cap\mathcal{T}^{\bot},$ $B\in\mathcal{X}$
and $\Ker(\varphi_{i})\in\mathcal{X}$ $\forall i\in[1,n]$. 
Let $M\in \T$.
By the dual of Lemma \ref{lem:recorrimiento de dimension}, 
$\Ext^k_\C(M,B)\cong\Ext^{n+k}_\C(M,A)\;\forall k\geq1,$ 
where $\Ext^{n+k}_\C(M,A)=0\;\forall k\geq1$ since $A\in\mathcal{X}$
and $\pd_\X(\T)\leq n.$ Therefore $B\in\mathcal{T}^{\bot}\cap\mathcal{X}.$
\end{proof}

We finish this section with an application in the $n$-cluster tilting theory. In order to do that, we generalize the notion of $n$-cluster tilting category, given by O. Iyama in \cite{iyama2011cluster}.

\begin{defi}\label{DefnXCT}  Let  $\X,\T\subseteq\C,$ $n\geq 1.$  We say that $\T$ is {\bf $n$-$\X$-cluster tilting} in  $\C$ if  the following conditions hold true.
\begin{itemize}
\item[$\mathrm{(a)}$] $\T=\add(\T).$
\item[$\mathrm{(b)}$] There exists $\alpha\subseteq \X^\perp\cap  \T^\perp$ which is a relative cogenerator in $\X.$
\item[$\mathrm{(c)}$] There exists $\beta\subseteq {}^\perp\X\cap {}^\perp \T$ which is a relative generator in $\X.$
\item[$\mathrm{(d)}$] $\T$ is functorially finite.
\item[$\mathrm{(e)}$] $\X\cap(\cap_{i=1}^{n-1}{}^{\perp_i}\T)=\T=\X\cap ( \cap_{i=1}^{n-1}\T^{\perp_i}).$
\end{itemize}
\end{defi}

 Let $\Lambda$ be a finite dimensional $k$-algebra and $\modu(\Lambda)$ be the category of finitely generated left 
$\Lambda$-modules. The notion of $n$-cluster tilting subcategory in $\modu(\Lambda)$ was introduced by O. Iyama in 
the study of a higher analogue of the classical Auslander correspondence between representation finite
algebras and Auslander algebras \cite{iyama2011cluster}. In this case, the term $n$-$\modu(\Lambda)$-cluster tilting matches with  Iyama's definition of $n$-cluster tilting subcategory in $\modu(\Lambda).$

The following corollary has been a key result in  \cite{Huerta2022periodic} to relate $n$-$\mathcal{X}$-cluster tilting with  $m$-periodic relative Gorenstein projective objects.

\begin{cor}\label{CorCTqs} For 
any
 $n$-$\X$-cluster tilting class $\T$ in the abelian category $\C,$ the following statements hold true.
\begin{itemize}
\item[$\mathrm{(a)}$] $\alpha,\beta\subseteq\T\subseteq\X.$
\item[$\mathrm{(b)}$] If $\pd_\X(\T)\leq n-1,$  then $\T=\X\cap \T^\perp =\X.$
\item[$\mathrm{(c)}$] If $\id_\X(\T)\leq n-1,$  then $\T=\X\cap{}^\perp\T =\X$.
\item[$\mathrm{(d)}$] If $\pd_\X(\X)\leq n-1,$  then $\X\cap \T^\perp=\T=\X\cap{}^\perp\T=\X$.
\end{itemize}
\end{cor} 
\begin{proof} (a)  If $n=1,$ then by Definition \ref{DefnXCT} 
(b,c,e)
 $\T=\X$ and there is nothing to prove. Let $n\geq 2.$ Then, by Definition \ref{DefnXCT} (b,e), 
  \[\alpha\subseteq\X^\perp\cap\X\cap  \T^\perp\subseteq \X\cap \cap_{i=1}^{n-1} \T^{\perp_i}= \T.\]
  Similarly, we also get that $\beta\subseteq\T.$
\

(b) Let $\pd_\X(\T)\leq n-1.$ If $n=1$ then $\T\subseteq\T^\perp$ and thus $\T=\T^\perp\cap\T.$ Moreover, by Definition \ref{DefnXCT} (e) 
$\T=\X$ and hence (b) holds true in this case.
\

Suppose that $n\geq 2.$ Since $\pd_\X(\T)\leq n-1,$ it follows from Definition \ref{DefnXCT} (e) that
$\X\cap\T^\perp=\X\cap\cap_{i=1}^{n-1}\T^{\perp_i}=\T.$ Furthermore, by  Proposition \ref{prop:M ortodonal cerrado por n-cocientes sii pdM=00003Dn} $\T ^{\bot }\cap \X$ is closed by $(n-1)$-quotients in $\X.$ Then, by using that $\beta\subseteq\T$ is a relative generator in $\X,$  we can conclude that $\X \subseteq \T.$
\

(c) It follows as in the proof of (b) by using the dual of Proposition \ref{prop:M ortodonal cerrado por n-cocientes sii pdM=00003Dn}.
\

(d) It  follows from (b) and (c) since $\pd_\X(\T)\leq \pd_\X(\X)$ and 
$\id_\X(\T)\leq \id_\X(\X)=\pd_\X(\X).$ 
\end{proof}

\section{Cotorsion-like pairs and related notions}

A main tool in the study of relative homological dimensions is the
notion of cotorsion pair which was introduced by Luigi Salce in \cite{salce}. Namely, a pair $(\A,\B)$ of classes of objects in an 
abelian category $\C$ is a cotorsion pair if $\A={}^{\perp_1}\B$ and $\B=\A^{\perp_1}.$

In this section, we  introduce a more general notion of relative cotorsion pair which will be useful for the  development of $n$-$\X$-tilting theory in \cite{Argudin-Mendoza2}.

\begin{defi}  Let $\A,\B,\X\subseteq\C.$ The pair $(\A,\B)$ is left (resp. right) cotorsion pair in $\X$ if $\A\cap\X={}^{\perp_1}\B\cap\X$ (resp. $\B\cap\X=\A^{\perp_1}\cap\X).$ In case both conditions hold true, we call $(\A,\B)$ cotorsion pair in $\X.$
\end{defi} 

Note that,  If $\p$ is a cotorsion pair in $\C,$ then $\p$ is a cotorsion pair in $\mathcal{X}$
for any $\mathcal{X}\subseteq\mathcal{C}$.
In case $\mathcal{X}=\mathcal{C},$ we simply say that $\p$ is a
left (right) cotorsion pair if it is a left (resp. right) cotorsion pair
in $\mathcal{X}.$  

\begin{ex} \label{exa:pares de cotorsion}
Let $\C$ be an abelian category and $\X\subseteq\C.$
\begin{enumerate}
\item[(1)] Let $\X$ be closed under extensions. The cotorsion pairs $\p$ in $\X$ satisfying that $\A,\B\subseteq\X$ were amply studied in \cite{ABsurvey} and called $\X$-cotorsion pairs. In particular, the authors proved that, for a left thick class $\mathcal{X}\subseteq\mathcal{C}$
(see Definition \ref{def: thick gruesa}), and an $\mathcal{X}$-injective relative
cogenerator $\omega=\smd(\omega) $ in $\mathcal{X},$ the pair
 $(\mathcal{X},\omega^{\wedge})$ is a cotorsion pair in $\mathcal{X}^{\wedge}$ \cite[Thm. 3.6]{ABsurvey}. 

\item[(2)]  The cotorsion pairs $\p$ in $\X$ satisfying that 
$\A=\smd(\A)$ and $\B=\smd(\B)$ were amply studied in \cite{huerta2020cut} and called cotorsion pairs cut along $\X.$ In this paper we can find several examples and applications. Namely, examples in the settings of relative Gorenstein homological algebra, chain complexes, and quasi-coherent sheaves; and as applications there are some results on the finitistic dimension conjecture, on the existence of right adjoints of quotient functors by Serre categories, and the description of cotorsion pairs in triangulated categories as co-t-structures
\end{enumerate}
\end{ex}

\subsection{Relative hereditary cotorsion-like pairs}

Let $\C$ be an abelian category, we recall that a pair $(\A,\B)$ in $\C$ is {\bf hereditary} if $\id_{\A}(\B)=0.$ Examples of such pairs are $({}{}^{\bot}\mathcal{Y},({}^{\bot}\mathcal{\mathcal{Y}})^{\bot})$
and $({}^{\bot}\left(\mathcal{\mathcal{Y}}^{\bot}\right),\mathcal{\mathcal{Y}}^{\bot}),$ for any $\mathcal{Y}\subseteq\mathcal{C}.$ 
Notice that such pairs are not necessarily cotorsion pairs. However, if $\C$ has enough projectives and injectives (see Corollary \ref{Tp=cp}), we have that the pairs $({}{}^{\bot}\mathcal{Y},({}^{\bot}\mathcal{\mathcal{Y}})^{\bot})$
and $({}^{\bot}\left(\mathcal{\mathcal{Y}}^{\bot}\right),\mathcal{\mathcal{Y}}^{\bot})$ are cotorsion pairs.  Hereditary
cotorsion pairs are an important tool for building resolutions and
coresolutions in homological algebra. Hence, it is of our interest
to use this kind of pairs in more general contexts. 
Unfortunately, if $\mathcal{C}$ is an abelian category without enough
projectives or injectives, we do not know if the hereditary pair $\p$ is a cotorsion
pair. Nonetheless, we will see in the following lines that, under
the proper hypotheses, $\p$ can be seen as a cotorsion pair inside
a class $\mathcal{X}\subseteq\mathcal{C}$ that admits an $\mathcal{X}$-injective
cogenerator and an $\mathcal{X}$-projective generator. In a way,
this will represent a relative notion of cotorsion pairs. 

\begin{defi}
Let $\mathcal{C}$ be an abelian category and $\mathcal{M},\mathcal{X}\subseteq\mathcal{C}$. 
\begin{itemize}
\item[$\mathrm{(a)}$] $\mathcal{M}$ is \textbf{closed under mono-cokernels in $\mathcal{M}\cap\mathcal{X}$}
if, for any exact sequence 
\[
\suc[M][M'][M'']\mbox{ with }M,M'\in\mathcal{M}\cap\mathcal{X}\mbox{,}
\]
 we have that $M''\in\mathcal{M}$. In case $\mathcal{M}\subseteq\mathcal{X}$,
we will simply say that $\mathcal{M}$ is \textbf{closed under mono-cokernels.
} The classes which are \textbf{closed under epi-kernels} are defined dually.
\item[$\mathrm{(b)}$] $\mathcal{M}$ is \textbf{$\mathcal{X}$-resolving} if $\mathcal{M}$
contains an $\mathcal{X}$-projective relative generator in $\mathcal{X}$, and $\mathcal{M}$
is closed under epi-kernels in $\mathcal{M}\cap\mathcal{X}$ and
 under extensions. The \textbf{$\mathcal{X}$-coresolving}  classes are defined dually.
\end{itemize}
\end{defi}

\begin{lem} \label{lem:buscando completar la ortogonal} Let  $\mathcal{X},\B\subseteq\mathcal{C}$ be classes such that
$\mathcal{B}$ is $\mathcal{X}$-coresolving. Then
\begin{center}
 $^{\bot_{1}}\left(\mathcal{B}\cap\mathcal{X}\right)\cap\mathcal{X}={}{}^{\bot}\left(\mathcal{B}\cap\mathcal{X}\right)\cap\mathcal{X}\mbox{.}$ 
 \end{center}
\end{lem}
\begin{proof}
It is enough to prove that $^{\bot_{1}}\left(\mathcal{B}\cap\mathcal{X}\right)\cap\mathcal{X}\subseteq{}{}^{\bot}\left(\mathcal{B}\cap\mathcal{X}\right)\cap\mathcal{X}$.
With this goal, we will show that $^{\bot_{k}}\left(\mathcal{B}\cap\mathcal{X}\right)\cap\mathcal{X}\subseteq{}{}^{\bot_{k+1}}\left(\mathcal{B}\cap\mathcal{X}\right)\cap\mathcal{X}\,\forall k>0\mbox{.}$
Let $C\in{}{}^{\bot_{k}}\left(\mathcal{B}\cap\mathcal{X}\right)\cap\mathcal{X}$ and 
$\alpha\subseteq\mathcal{B}\cap\X$ be an $\mathcal{X}$-injective
relative cogenerator in $\mathcal{X}$. Hence, any $M\in\mathcal{B}\cap\mathcal{X}$
admits an exact sequence $\eta:\quad\suc[M][A][K],$
with $A\in\alpha\subseteq\mathcal{B}\cap\mathcal{X}$ and $K\in\mathcal{X}$.
Observe that $K\in\mathcal{\mathcal{B}}\cap\mathcal{X}$ since $\mathcal{\mathcal{B}}$
is closed under mono-cokernels in $\mathcal{B}\cap\mathcal{X}$. Hence,
in the exact sequence $\Extx[k][][C][K]\rightarrow\Extx[k+1][][C][M]\rightarrow\Extx[k+1][][C][A],$
we have $\Extx[k+1][][C][A]=0$ and $\Extx[k][][C][K]=0.$ Thus $\Extx[k+1][][C][\mathcal{B}\cap\mathcal{X}]=0.$
\end{proof}

As a direct consequence of Lemma \ref{lem:buscando completar la ortogonal}, we get the following corollary. 

\begin{cor} \label{lem:Garcia-Rozas}  Let $\A,\B,\mathcal{X}\subseteq\mathcal{C}$
be a classes such that $\Extx[1][][\mathcal{A}\cap\mathcal{X}][\mathcal{B}\cap\mathcal{X}]=0$.
If $\mathcal{B}$ is $\mathcal{X}$-coresolving and $\mathcal{A}$
is closed under epi-kernels in $\mathcal{A}\cap\mathcal{X}$, then for
any exact sequence $\suc[K][A][A']\mbox{,}$ with $A,A'\in\mathcal{A}\cap\mathcal{X}$
and $K\in\mathcal{X}$, we have $K\in{}{}^{\bot}\left(\mathcal{B}\cap\mathcal{X}\right)$.
\end{cor}

\begin{lem}\label{lem:Garcia-Rozas-1} Let $\A,\B,\mathcal{X}\subseteq\mathcal{C}$
be  classes such that $\Extx[1][][\mathcal{A}\cap\mathcal{X}][\mathcal{B}\cap\mathcal{X}]=0$ and $\A$ contains a class which is an $\mathcal{X}$-projective relative
generator in  $\mathcal{X}.$ Then, the following statements are equivalent:
\begin{enumerate}
\item[$\mathrm{(a)}$] For any exact sequence $\suc[K][A][A']$, with $A,A'\in\mathcal{A}\cap\mathcal{X}$
and $K\in\mathcal{X}$, we have that $K\in{}{}^{\bot}\left(\mathcal{B}\cap\mathcal{X}\right).$
\item[$\mathrm{(b)}$] $\idr[\mathcal{A}\cap\mathcal{X}][\mathcal{B}\cap\mathcal{X}]=0$.
\end{enumerate}
\end{lem}
\begin{proof}
We only prove that (a) implies (b) since the other implication is quite similar. Let $A\in\mathcal{A}\cap\mathcal{X}$ and $B\in\mathcal{B}\cap\mathcal{X}$. Since $\A$ contains an $\X$-projective relative generator in $\X$, there is an exact sequence  $\suc[C][P][A]$, with $C\in\mathcal{X}$ and $P\in\mathcal{A}\cap\mathcal{X}\cap {}^{\perp}\mathcal{X}$. Moreover, we have that $C\in{}{}^{\bot}\left(\mathcal{B}\cap\mathcal{X}\right)$ by (a). Then, from the exact sequence \[ \Extx[n-1][][C][B]\rightarrow\Extx[n][][A][B]\rightarrow\Extx[n][][P][B]\mbox{,} \] it follows that $\Extx[n][][A][B]=0\,\forall n\geq2$. Finally, $\Extx[1][][A][B]=0$ holds true by hypothesis.
\end{proof}

\begin{defn}
Let $\mathcal{C}$ be an abelian category and $\mathcal{X}\subseteq\mathcal{C}$.
A pair $\p$ in $\C$ is \textbf{$\mathcal{X}$-hereditary}
if $\idr[\mathcal{A}\cap\mathcal{X}][\mathcal{B}\cap\mathcal{X}]=0$.
In case that $\mathcal{X}=\mathcal{C}$, we will simply say that $\p$
is \textbf{hereditary}.
\end{defn}

Let $\mathcal{Y}\subseteq\mathcal{C}.$ As was mentioned before,  we do
not know if $\p[^{\bot}(\mathcal{Y}^{\bot})][\mathcal{Y}^{\bot}]$
is a cotorsion pair.   In the next Lemmas we discuss this phenomenon.
We will see that, under some conditions related to a given class
$\mathcal{X}$, it is possible to find a cotorsion pair $\p$ such
that $\mathcal{B}\cap\mathcal{X}=\mathcal{Y}^{\bot}\cap\mathcal{X}$,
or such that $\mathcal{A}\cap\mathcal{X}={}^{\bot}(\mathcal{Y}^{\bot})\cap\mathcal{X}$.

\begin{lem} \label{lem:cotorsion completa} Let  $\mathcal{X},\mathcal{Y}\subseteq\mathcal{C}$ be such that there
is an $\mathcal{X}$-injective relative cogenerator in $\mathcal{Y}.$
Then,  $^{\bot}\mathcal{Y}\cap\mathcal{\mathcal{X}}={}^{\bot_{1}}\mathcal{Y}\cap\mathcal{X}\mbox{.}$\end{lem}
\begin{proof}
Let $\alpha$ be an $\mathcal{X}$-injective relative cogenerator
in $\mathcal{Y}$ and $M\in\mathcal{Y}$.  Then, there is an exact sequence
$0\rightarrow M\rightarrow A_{0}\rightarrow\cdots\rightarrow A_{n}\rightarrow A_{n+1}\rightarrow\cdots$ with $A_{i}\in\alpha$ and $M_{i+1}:=\im{A_{i}\rightarrow A_{i+1}}\in\mathcal{Y}$
$\forall i\geq0$. Consider the set $\mathcal{N}_{M}:=\left\{ M_{i}\right\} {}_{i=0}^{\infty}$,
where $M_{0}=M$. Since $\alpha$ is $\mathcal{X}$-injective, by
the 
Lemma \ref{lem:recorrimiento de dimension},
 we have $\Extx[1][][X][M_{i}]\cong\Extx[1+i][][X][M]\forall X\in\mathcal{X}\,,\forall i\geq0\mbox{.}$ Therefore, $^{\bot}M\cap\mathcal{\mathcal{X}}={}^{\bot_{1}}\mathcal{N}_{M}\cap\mathcal{X}$.
Finally, note that $\mathcal{Y}=\bigcup_{M\in\mathcal{Y}}\mathcal{N}_{M}$
and hence ${}^{\bot_{1}}\mathcal{Y}\cap\mathcal{X}={}^{\bot}\mathcal{Y}\cap\mathcal{\mathcal{X}}$. 
\end{proof}

\begin{pro} Let $\mathcal{Y}\subseteq\mathcal{X}\subseteq\mathcal{C}$ be classes such that
there is an $\mathcal{X}$-injective
relative cogenerator in $\mathcal{Y}$ and there is an $\mathcal{X}$-projective
relative generator in $\mathcal{X}$ contained in $^{\bot}\mathcal{Y}.$ Then,
 $({}{}^{\bot}\mathcal{Y}\cap\mathcal{X},\left({}{}^{\bot}\mathcal{Y}\cap\mathcal{X}\right)^{\bot}\cap\mathcal{X})$
is a cotorsion pair in $\mathcal{X}$.
\end{pro}
\begin{proof}
By Lemma \ref{lem:cotorsion completa},  $^{\bot}\mathcal{Y}\cap\mathcal{\mathcal{X}}={}^{\bot_{1}}\mathcal{Y}\cap\mathcal{X}.$
Since $^{\bot}\mathcal{Y}$ is $\mathcal{X}$-resolving, by the dual of Lemma \ref{lem:buscando completar la ortogonal},
we have 
\[
\left({}{}^{\bot}\mathcal{Y}\cap\mathcal{X}\right)^{\bot}\cap\mathcal{X}=\left({}{}^{\bot}\mathcal{Y}\cap\mathcal{X}\right)^{\bot_{1}}\cap\mathcal{X}=\left({}{}^{\bot_{1}}\mathcal{Y}\cap\mathcal{X}\right)^{\bot_{1}}\cap\mathcal{X}\mbox{.}
\]
Therefore, 
$\mathfrak{C}:=({}{}^{\bot}\mathcal{Y}\cap\mathcal{X},\left({}{}^{\bot}\mathcal{Y}\cap\mathcal{X}\right)^{\bot}\cap \X)$
is a right cotorsion pair in $\mathcal{X}$. It remains to prove
that $\mathfrak{C}$ is a left cotorsion pair in $\mathcal{X}$. To
do this, we must show that\\ 
$
^{\bot_{1}}\left(\left({}{}^{\bot_{1}}\mathcal{Y}\cap\mathcal{X}\right)^{\bot_{1}}\cap\mathcal{X}\right)\cap\mathcal{X}={}^{\bot_{1}}\mathcal{Y}\cap\mathcal{X}\mbox{.}
$
Indeed, note that 
\[
^{\bot_{1}}\left(\left({}{}^{\bot_{1}}\mathcal{Y}\cap\mathcal{X}\right)^{\bot_{1}}\cap\mathcal{X}\right)\cap\mathcal{X}\subseteq{}^{\bot_{1}}\left(\left({}{}^{\bot_{1}}\mathcal{Y}\right)^{\bot_{1}}\cap\mathcal{X}\right)\cap\mathcal{X}\subseteq{}^{\bot_{1}}\left(\mathcal{Y}\cap\mathcal{X}\right)\cap\mathcal{X}={}^{\bot_{1}}\mathcal{Y}\cap\mathcal{X}\mbox{,}
\]
and  $^{\bot_{1}}\mathcal{Y}\cap\mathcal{X}\subseteq{}{}^{\bot_{1}}\left(\left({}{}^{\bot_{1}}\mathcal{Y}\cap\mathcal{X}\right)^{\bot_{1}}\right)\cap\mathcal{X}\subseteq{}{}^{\bot_{1}}\left(\left({}{}^{\bot_{1}}\mathcal{Y}\cap\mathcal{X}\right)^{\bot_{1}}\cap\mathcal{X}\right)\cap\mathcal{X}.$
\end{proof}

In case $\mathcal{Y}=\{T\}$, we can claim the following statements. For $\omega\subseteq\C,$ we denote by $\omega_{\infty}^{\wedge}$ the class of all the objects $C\in\C$ admitting an infinite exact sequence $\cdots\rightarrow W_{n}\rightarrow\cdots\rightarrow W_{1}\rightarrow W_{0}\rightarrow C\rightarrow0\mbox{,}$ with $W_{i}\in\omega\cup\left\{ 0\right\}$ $\forall\, i\geq 0.$ We denote by $\omega^{\wedge}$ the class of all the objects $C\in\C$ admitting a finite exact sequence $0\rightarrow W_{n}\rightarrow\cdots\rightarrow W_{1}\rightarrow W_{0}\rightarrow C\rightarrow0\mbox{,}$ with $W_{i}\in\omega$ $\forall\, i\in[0,n].$ 
Now, we consider the following result whose idea 
comes
 from M. Auslander.

\begin{lem} \label{lem:lema de auslander} For
$T\in\mathcal{C}$ and $\omega,\mathcal{X}\subseteq\mathcal{C}$ such that $\omega$ is $\mathcal{X}$-projective, 
the following statements hold true.
\begin{enumerate}
\item[$\mathrm{(a)}$] if $T\in\omega_{\infty}^{\wedge}$, then there is a set $\mathcal{S}\subseteq\mathcal{C}$
such that $T^{\bot}\cap\mathcal{\mathcal{X}}=\mathcal{S}{}^{\bot_{1}}\cap\mathcal{X}.$
\item[$\mathrm{(b)}$] if $T\in\omega^{\wedge}$, then there is an object $S\in\mathcal{C}$
such that $T^{\bot}\cap\mathcal{\mathcal{X}}=S{}^{\bot_{1}}\cap\mathcal{X}.$
\item[$\mathrm{(c)}$] if $\mathcal{C}$ is an AB4 category and $T\in\omega_{\infty}^{\wedge}$,
then there is an object $S\in\mathcal{C}$ such that $T^{\bot}\cap\mathcal{\mathcal{X}}=S{}^{\bot_{1}}\cap\mathcal{X}.$
\end{enumerate}
\end{lem}
\begin{proof}
Let $T\in\omega_{\infty}^{\wedge}.$ Then, there is
an exact sequence 
\[
\cdots\rightarrow W_{n+1}\rightarrow W_{n}\rightarrow\cdots\rightarrow W_{1}\rightarrow W_{0}\rightarrow T\rightarrow0\mbox{,}
\]
with $W_{i}\in\omega\cup\left\{ 0\right\} $ and $T_{i+1}:=\im{W_{i+1}\rightarrow W_{i}}\,\forall i\geq0$.
Consider $\mathcal{S}:=\left\{ T_{i}\right\} {}_{i=0}^{\infty}$,
where $T_{0}:=T$. Since $\omega\subseteq{}{}^{\bot}\mathcal{X}$,
by the Shifting Lemma we have 
\[
\Extx[][][T][X]\cong\Extx[1][][T_{i}][X]\forall i\geq0,\forall X\in\mathcal{X},
\]
and thus $\mathcal{S}^{\bot_{1}}\cap\mathcal{X}=T^{\bot}\cap\mathcal{X};$ proving (a). If $\C$ is AB4, by 
Theorem \ref{prop:extn coprods arb}, the coproduct $S:=\bigoplus_{i=0}^{\infty}T_{i}$
satisfies that $S^{\bot_{1}}=\mathcal{S}^{\bot_{1}}$ and thus (c) holds true. The proof of (b) is quite similar to the one given of (c) since in this case the set $\mathcal{S}$ is finite.
\end{proof}

The following is a particular case when the pair $({}^{\perp}(\T^\perp),\T^\perp)$ is a cotorsion pair, for a class $\T\subseteq\C.$ 

\begin{cor}\label{Tp=cp} Let $\C$ be an abelian  category  with enough projectives  and injectives, and let $\T\subseteq\C.$ Then $({}^{\perp}(\T^\perp),\T^\perp)$ is a hereditary cotorsion pair in $\C.$
\end{cor}
\begin{proof} Let $T\in \T.$ Since $\C$ has enough projectives, by Lemma \ref{lem:lema de auslander} (a), there is a set $S_T\subseteq\C$ such that $T^\perp=S^{\perp_1}_T.$
Now, consider the class $S:=\bigcup_{T\in\T} S_T.$ Then $S^{\perp_1}=\bigcap_{T\in\T} S^{\perp_1}_T=\bigcap_{T\in\T} T^\perp=\T^\perp.$ Finally, since $\C$ has enough injectives, by Lemma \ref{lem:buscando completar la ortogonal}, we get that 
${}^{\perp_1}(S^{\perp_1})={}^{\perp_1}(\T^{\perp})={}^{\perp}(\T^{\perp})$ and thus the result 
follows since ${}({}^{\perp_1}(S^{\perp_1}))^{\perp_1}=S^{\perp_1}$.
\end{proof}

In the situation presented in the previous result, it is well known that the pair $({}^{\perp}(\T^\perp),\T^\perp)$ is complete in case $S$ is a set. That is, in such a case, the pair induces ${}^{\perp}(\T^\perp)$-precovers and $\T^\perp$-preenvelopes. In the next section we will study these concepts.

\subsection{Relative complete pairs}

We will be using the following notation and vocabulary for approximations. We start this section by recalling the well known notions of precovers and preenvelopes.
\

For a class $\Z$ of objects in an abelian category $\C,$ a morphism $f:Z\rightarrow M$ in $\C$ is called a \textbf{$\mathcal{Z}$-precover}
if $Z\in\mathcal{Z}$ and 
$\Homx[][Z'][f]:\Homx[\mathcal{C}][Z'][Z]\rightarrow\Homx[\mathcal{C}][Z'][M]$
 is surjective $\forall Z'\in\mathcal{Z}.$ A $\mathcal{Z}$-precover $Z\rightarrow M$ is called \textbf{special}
if it fits in an exact sequence $\suc[M'][Z][M][\,][\,]\mbox{, where }M'\in\mathcal{Z}^{\bot_{1}}\mbox{.}$ It is said that $\Z$ is {\bf precovering} if each $C\in\C$ admits a  $\mathcal{Z}$-precover $Z\rightarrow C.$ The notions of (special) $\Z$-preenvelope and preenveloping class are defined dually. The class $\Z$ is called {\bf functorially finite} if it is precovering and preenveloping.
In the case of an Artin algebra $\Lambda,$ 
a precovering (resp. preenveloping) class in $\modd[\Lambda]$ is usually called contravariantly (resp. covariantly)
finite.
\

In order to develop the relative $n$-$\X$tilting theory in \cite{Argudin-Mendoza2}, we need a relative version of special precover (preenvelope). We elaborate these ideas below.

\begin{defn} Let $\mathcal{C}$ be an abelian category and $\mathcal{Z},\mathcal{X}\subseteq\mathcal{C}.$ We say that 
$\mathcal{Z}$ is \textbf{special precovering in $\mathcal{X}$} if
any $X\in\mathcal{X}$ admits an exact sequence 
\[
\suc[B][A][X]\mbox{, with \ensuremath{A}\ensuremath{\in\mathcal{Z}\cap\mathcal{X}} and \ensuremath{B}\ensuremath{\in\mathcal{Z}^{\bot_{1}}\cap\mathcal{X}} .}
\]
The notion of \textbf{special preenveloping in $\mathcal{X}$} is defined dually. 
\end{defn}

The previous notions are related with the notion of \emph{complete
cotorsion pair}. Let us present the following definitions that generalize
such notion. 

\begin{defn}
Let $\mathcal{C}$ be an abelian category and $\A,\B,\X\subseteq\mathcal{C}.$
The pair $\p$ is \textbf{left $\mathcal{X}$-complete} if any $X\in\mathcal{X}$
admits an exact sequence $\suc[B][A][X]$, with $A\in\mathcal{A}\cap\mathcal{X}$
and $B\in\mathcal{B}\cap\mathcal{X}$. The notion of  \textbf{right $\mathcal{X}$-complete} pair is defined dually. We say that $\p$ is \textbf{$\mathcal{X}$-complete}
if it is right and left $\mathcal{X}$-complete. 
\end{defn}

The Salce's Lemma 
(see \cite[Cor. 2.4]{salce})
 can be written, with the obvious proof, as follows.

\begin{lem}\label{lem:Lema-de-Salce} Let 
$\mathcal{X}\subseteq\mathcal{C}$ be closed under extensions and $\p$ 
be a left $\mathcal{X}$-complete pair in $\C$ such that $\B$ is
closed under extensions and contains a relative cogenerator
in $\mathcal{X}$. Then $\p$ is right $\mathcal{X}$-complete.
\end{lem}

\begin{example}\label{exa:cot pair}\label{exa:cluster,n-pres} Let $\C$ be an abelian category  and $\mathcal{S}\subseteq\mathcal{C}.$ A left cotorsion pair $\p$ cut along 
$\mathcal{S}$ is called complete if for each $S\in\mathcal{S}$ there is an exact sequence\\ 
$\suc[B][A][S][\,][\,]\mbox{,}$ 
with $A\in\mathcal{A}$ and $B\in\mathcal{B}$ \cite[Def. 2.1]{huerta2020cut}. The notions of relative completeness and cotorsion settled in \cite{huerta2020cut} and ours are a little different, but they are related as follows:
\begin{enumerate}
\item[$\mathrm{(a)}$] Every left $\mathcal{S}$-complete and left cotorsion
pair $\p$ in $\mathcal{S}$ with $\mathcal{A}=\smdx[\mathcal{A}]$ is a complete left
cotorsion pair cut along $\mathcal{S}.$ 
\item[$\mathrm{(b)}$] Let $\mathcal{A},\mathcal{B}\subseteq\mathcal{S}$. If $\p$ is a complete left cotorsion
pair cut along $\mathcal{S},$ then $\p$ is
left $\mathcal{S}$-complete and a left cotorsion pair in $\mathcal{S}.$ 
\end{enumerate}
\end{example}

\section{Relative resolutions, coresolutions and related classes}

Let us introduce the type of relative resolutions and coresolutions that 
are of interest to us
and that will be useful for the development of the theory related with some special classes and properties of the relative homological dimensions. This treatment will be necessary for $n$-$\X$-tilting theory in \cite{Argudin-Mendoza2}.

\begin{defn}\label{def: res cores}
Let $\mathcal{C}$ be an abelian category, $M\in\mathcal{C}$ and $\X,\Y,\Z\subseteq\mathcal{C}.$
\begin{enumerate}
\item[$\mathrm{(a)}$] An exact sequence $0\rightarrow M\stackrel{f_{0}}{\rightarrow}Y_{0}\stackrel{f_{1}}{\rightarrow}Y_{1}\stackrel{}{\rightarrow}\cdots\stackrel{}{\rightarrow}Y_{n-1}\stackrel{f_{n}}{\rightarrow}Y_{n}\stackrel{}{\rightarrow}\cdots$ in $\C,$ with
$Y_{k}\in\mathcal{Y}\cup\left\{ 0\right\} $ $\forall k\geq0$ and
$\im{f_{i}}\in\mathcal{X}\cup\{0\}$ $\forall i\geq1$, is called
a \textbf{$\mathcal{Y}_{\mathcal{X}}$-coresolution} of $M$.

\item[$\mathrm{(b)}$]  An exact sequence $0\rightarrow M\stackrel{f_{0}}{\rightarrow}Y_{0}\stackrel{f_{1}}{\rightarrow}Y_{1}\stackrel{}{\rightarrow}\cdots\stackrel{}{\rightarrow}Y_{n-1}\stackrel{f_{n}}{\rightarrow}Y_{n}\stackrel{}{\rightarrow}0$ in $\C,$ with
$Y_{n}\in\mathcal{X}\cap\mathcal{Y}$, $Y_{k}\in\mathcal{Y}$ $\forall k\in[0,n-1]$
and $\im{f_{i}}\in\mathcal{X},$ $\forall i\in[1,n-1]$, is called
a\textbf{ $\mathcal{Y}_{\mathcal{X}}$-coresolution of length $n$}
of $M$, or simply a \textbf{finite $\mathcal{Y}_{\mathcal{X}}$-coresolution
of $M$.}

\item[$\mathrm{(c)}$]  We define $\coresdimr{\mathcal{Y}}M{\mathcal{X}},$ the \textbf{$\mathcal{Y}_{\mathcal{X}}$-coresolution dimension}
of $M,$ which is the smallest non-negative integer $n$ such that there is a $\mathcal{Y}_{\mathcal{X}}$-coresolution
of length $n$ of $M$. If such $n$ does not exist, we set $\coresdimr{\mathcal{Y}}M{\mathcal{X}}:=\infty\mbox{.}$

\item[$\mathrm{(d)}$]  The $\mathcal{Y}_{\mathcal{X}}$-coresolution dimension of the class $\mathcal{Z}$
is defined as 
\[
\coresdimr{\mathcal{Y}}{\mathcal{Z}}{\mathcal{X}}:=\sup\left\{ \coresdimr{\mathcal{Y}}Z{\mathcal{X}}\,|\:Z\in\mathcal{Z}\right\} \mbox{.}
\]

\item[$\mathrm{(e)}$]  We denote by $\mathcal{Y}_{\mathcal{X},\infty}^{\vee}$ $(\text{resp. }\mathcal{Y}_{\mathcal{X}}^{\vee})$ the
class of all the objects in $\mathcal{C}$ having a (resp. finite) $\mathcal{Y}_{\mathcal{X}}$-coresolution. 
Notice that $\mathcal{Y}_{\mathcal{X}}^{\vee}\subseteq\mathcal{Y}_{\mathcal{X},\infty}^{\vee}$.

\item[$\mathrm{(f)}$]  We denote by $\mathcal{Y}_{\mathcal{X},n}^{\vee}$ the class
of all the objects in $\mathcal{C}$ having a $\mathcal{Y}_{\mathcal{X}}$-coresolution
of length $\leq n$.

\item[$\mathrm{(g)}$]  $\p[\mathcal{X}][\mathcal{Y}]_{\infty}^{\vee}:=\mathcal{X}\cap\mathcal{Y}_{\mathcal{X},\infty}^{\vee},$ $\p[\mathcal{X}][\mathcal{Y}]^{\vee}:=\mathcal{X}\cap\mathcal{Y}_{\mathcal{X}}^{\vee}$
and $\p[\mathcal{X}][\mathcal{Y}]_{n}^{\vee}:=\mathcal{X}\cap\mathcal{Y}_{\mathcal{X},n}^{\vee}.$

\item[$\mathrm{(h)}$]  Dually, we define the $\mathcal{Y}_{\mathcal{X}}$-resolution $(\text{of length $n$}),$  the $\mathcal{Y}_{\mathcal{X}}$-resolution dimension $\resdim_\Y^{\X}(M)$ of $M$ and the classes $\mathcal{Y}_{\mathcal{X}}^{\wedge},$  $\mathcal{Y}_{\mathcal{X},\infty}^{\wedge}$ and
$\mathcal{Y}_{\mathcal{X},n}^{\wedge}.$ We also have 
$\p[\mathcal{Y}][\mathcal{X}]_{\infty}^{\wedge}:=\mathcal{Y}_{\mathcal{X},\infty}^{\wedge}\cap\mathcal{X},$
$\p[\mathcal{Y}][\mathcal{X}]^{\wedge}:=\mathcal{Y}_{\mathcal{X}}^{\wedge}\cap\mathcal{X}$
and $\p[\mathcal{Y}][\mathcal{X}]_{n}^{\wedge}:=\mathcal{Y}_{n}^{\wedge}\cap\mathcal{X}.$
\end{enumerate}
 If $\mathcal{X}=\mathcal{C},$ we omit the symbol ``$\mathcal{X}$'' 
in the above notations.
Note that
\begin{center}
 $\resdimr{\mathcal{Y}}M{\mathcal{X}}=0$  $\Leftrightarrow$ $M\simeq \mathcal{X}\cap\mathcal{Y}$ $\Leftrightarrow$ $\coresdimr{\mathcal{Y}}M{\mathcal{X}}=0.$
\end{center}
\end{defn}

Some of the above relative resolutions have been considered before in several papers. In what follows, we consider some examples where such resolutions appear.

\begin{example} (1) In \cite[Sect. 5]{auslandereiten}  M.
Auslander and I. Reiten 
were interested in the study of the contravariantly finite subcategories in $\modu(\Lambda)$
induced by \emph{cotilting modules} over an Artin algebra $\Lambda.$ In particular, they studied the following classes $\mathcal{X}_{\omega}:=(\addx[\omega]){}_{^{\bot}\omega,\infty}^{\vee}$ and 
$_{\omega}\mathcal{X}:=(\addx[\omega])_{\omega^{\bot},\infty}^{\wedge},$ for a class 
$\omega\subseteq\modd[\Lambda]$ such that
$\omega\subseteq\omega^{\bot}.$ 
\

(2) In \cite[Def. 3.11]{relgor}, the authors studied the  relative Gorenstein objects in abelian categories. Namely, for 
an abelian category $\C$ and $\A,\B\subseteq\mathcal{C},$ they introduced, respectively, the class of 
the weak $\p[\mathcal{A}][\mathcal{B}]$-Gorenstein
projectives  (resp. injectives) objects
$\mathcal{WGP}_{\p[\mathcal{A}][\mathcal{B}]}:=({}^{\bot}\mathcal{B},\mathcal{A})_{\infty}^{\vee}\quad$ 
$\quad(\text{resp. }\mathcal{WGI}_{\p[\mathcal{A}][\mathcal{B}]}:=(\mathcal{A},\mathcal{B}^{\bot})_{\infty}^{\wedge}).$
\

(3) Let $\mathcal{C}$ be an $n$-coherent category \cite[Def. 4.1]{bravo2019locally} .
An object $M\in\mathcal{C}$ is Gorenstein $\mathcal{FP}_{n}$-injective
if there is an exact sequence
\[
\eta:\:\cdots\rightarrow I_{2}\rightarrow I_{1}\rightarrow I_{0}\overset{f}{\rightarrow}E_{0}\rightarrow E_{1}\rightarrow E_{2}\rightarrow\cdots
\]
with $I_{i},E_{i}\in\Injx[\mathcal{C}]$ $\forall i\geq1$ such that
$M=\im{f}$ and the complex $\Homx[\mathcal{C}][J][\eta]$ is acyclic for any
$J\in\mathcal{FP}_{n}^{\bot_{1}}$. Let $\mathcal{GI}$ denote the
class of all the Gorenstein $\mathcal{FP}_{n}$-injective objects and $\mathcal{W}:={}^{\bot_{1}}\mathcal{GI}$.
In \cite[Lem. 5.2]{bravo2019locally}, it is proved that $\mathcal{GI}=(\Injx[\mathcal{C}],\mathcal{G}\mathcal{I})_{\infty}^{\wedge}$. 
\

(4) The above example is a particular case of the following observation.
Let $\mathcal{X}$ be a class in an abelian category $\mathcal{C}$
and $0\in\omega\subseteq\mathcal{X}$. Then, $\mathcal{X}=(\omega,\mathcal{X})_{\infty}^{\wedge}$ if, and only
if, $\omega$ is a relative generator in $\mathcal{X}$. 
\end{example}

We recall the following known 
lemma. 

\begin{lem} \label{lem:dims vs res cores}\cite[Lem. 2.13(a)]{tiltstrat}\label{lem:(4.10homologia_relativa)}
Let $\mathcal{C}$ be an abelian category and $\mathcal{X},\mathcal{Y}\subseteq\mathcal{C}$.
Then, $\pdr[\mathcal{Y}][\mathcal{X}^{\vee}]=\pdr[\mathcal{Y}][\mathcal{X}]$.
\end{lem}

The next theorem shows the existence of the main approximations that we will be using
in the relative class $\,\U_{\V}^{\vee}$. This theorem is
a generalization of the dual 
 of \cite[Thm. 1.1]{Auslander-Buchweitz} and will play an important role in the development of $n$-$\X$ tilting theory in \cite{Argudin-Mendoza2}. For example, it is fundamental in the proof that the pair $({}^\perp(\T^\perp),\T^\perp)$ is $\X$-complete if $\T$ is an $n$-$\X$-tilting class in $\C.$

 It is worth  mentioning that the statement of Theorem \ref{thm:(5.4homologia_relativa)}  is inspired by  \cite[Thm. 2.8]{ABsurvey}, where the authors present 
 Auslander-Buchweitz results with the
minimum hypotheses needed. 

In the following result, the expression $\coresdim_\omega^{\V}(C_Z)=-1$
just means that $C_Z=0$.

\begin{thm} \label{thm:(5.4homologia_relativa)}\label{thm:(5.4homologia_relativa)-1}
For $\U\subseteq\V\subseteq\mathcal{C}$  classes which are closed under extensions, $\omega$  a relative generator in $\U$ and $0\in\U,$   the following statements
hold true.
\begin{enumerate}
\item[$\mathrm{(a)}$]  For any $Z\in\U_{\V}^{\vee}$, with $n:=\coresdimr{\U}Z{\V}$,
there are short exact sequences 
\begin{alignat*}{1}
\suc[Z][M_{Z}][C_{Z}][g_{Z}] & \mbox{ with \ensuremath{C_{Z}\in(\V,\omega){}^{\vee},\,M_{Z}\in\U}}\mbox{ and }\\
\suc[K_{Z}][B_{Z}][Z][][f_{Z}] & \mbox{ with \ensuremath{B_{Z}\in\omega_{\V}^{\vee}}, \ensuremath{K_{Z}\in\U}}\mbox{,}
\end{alignat*}
where $\coresdimr{\omega}{C_{Z}}{\V}=n-1$ and $\coresdimr{\omega}{B_{Z}}{\V}\leq n$. 

\item[$\mathrm{(b)}$]  $B_{Z}\in(\V,\omega){}^{\vee}$ if $Z\in(\V,\U){}^{\vee}.$

\item[$\mathrm{(c)}$]  Let $\omega\subseteq{}{}^{\bot}\U.$ Then $\omega^{\vee}\subseteq{}{}^{\bot}\U$,
$f_{Z}$ is a $\omega^{\vee}$-precover, and $g_{Z}$ is an $\U$-preenvelope.
\end{enumerate}
\end{thm}
\begin{proof} (a) We proceed by induction on $n=\coresdimr{\U}Z{\V}=n<\infty$. 
\

Let $n=0$. Then, $Z\in\U\cap\V$. Now, since $\omega$
is a relative generator in $\U$, there is an exact sequence
$\suc[U'][W][Z]\mbox{,}$ with $W\in\omega\subseteq\p[\V][\omega]^{\vee}$
and $U'\in\U$. We can also consider the exact sequence $\suc[Z][Z][0][1]\mbox{,}$
where $0\in\p[\V][\omega]^{\vee}$. Note that these are the
desired exact sequences.\\
Let $n>0$. Hence, there is an exact sequence $\suc[Z][U_{0}][V][\,][f]$,
with $U_{0}\in\U$,\\
\begin{minipage}[t]{0.6\columnwidth}%
$V\in\V$, and $\coresdimr{\U}V{\V}=n-1$.
Now, by inductive hypothesis, there is a short exact sequence 
\[
\suc[K_{V}][B_{V}][V][\,][f_{V}]
\]
with $\ensuremath{B_{V}\in\omega_{\V}^{\vee}}$, $\ensuremath{K_{V}\in\U\subseteq\Y}$,
and $\coresdimr{\omega}{B_{V}}{\V}\leq n-1$. Considering
the pull-back of $f_{V}$ and $f$, we get a short exact sequence 
\[
\eta:\quad\suc[Z][E][B_{V}][h][\,]
\]
with $E\in\U$ and $B_{V}\in\V$. Note that $\eta$ is the first of the two exact sequences we are looking for.\\
Now, since $E\in\U$,
there is an exact sequence %
\end{minipage}\hfill{}%
{\begin{minipage}[t]{0.35\columnwidth}%
\[
\begin{tikzpicture}[-,>=to,shorten >=1pt,auto,node distance=.9cm,main node/.style=,x=1.5cm,y=1.5cm,framed]

   \node[main node] (C) at (0,0)      {$V$};
   \node[main node] (X0) [left of=C]  {$U_0$};
   \node[main node] (X1) [left of=X0]  {$Z$};

   \node[main node] (X) [above of=C]  {$B_{V}$};
   \node[main node] (E) [left of=X]  {$E$};
   \node[main node] (X2) [left of=E]  {$Z$};

   \node[main node] (Y1) [above of=X]  {$K_{V}$};
   \node[main node] (Y2) [above of=E]  {$K_{V}$};

   \node[main node] (01) [below of=C]    {$0$};
   \node[main node] (02) [below of=X0]    {$0$};

   \node[main node] (03) [right of=C]    {$0$};
   \node[main node] (04) [left of=X1]    {$0$};

   \node[main node] (05) [right of=X]    {$0$};
   \node[main node] (06) [left of=X2]    {$0$};

   \node[main node] (07) [above of=Y1]    {$0$};
   \node[main node] (08) [above of=Y2]    {$0$};

\draw[->, thick]   (X0)  to node  {$$}  (C);
\draw[->, thick]   (X1)  to node  {$$}  (X0);
\draw[->, thick]   (C)  to node  {$$}  (03);
\draw[->, thick]   (04)  to node  {$$}  (X1);

\draw[->, thick]   (Y1)  to node  {$$}  (X);
\draw[->, thin]   (Y2)  to node  {$$}  (E);
\draw[->, thick]   (X)  to node  {$$}  (C);
\draw[-, double]   (X1)  to node  {$$}  (X2);
\draw[->, thin]   (E)  to node  {$$}  (X0);

\draw[->, thin]   (X)  to node  {$$}  (05);
\draw[->, thin]   (E)  to  node  {$$}  (X);
\draw[->, thin]   (X2)  to  node   {$$}  (E);
\draw[->, thin]   (06)  to  node   {$$}  (X2);

\draw[-, double]   (Y2)  to  node  {$$}  (Y1);
\draw[->, thick]   (07)  to  node   {$$}  (Y1);
\draw[->, thin]   (08)  to  node   {$$}  (Y2);
\draw[->, thick]   (C)  to  node   {$$}  (01);
\draw[->, thin]   (X0)  to  node   {$$}  (02);
   
\end{tikzpicture}
\]%
\end{minipage}}\\
{\begin{minipage}[t]{0.35\columnwidth}%
\[
\begin{tikzpicture}[-,>=to,shorten >=1pt,auto,node distance=1cm,main node/.style=,x=.45cm,y=.45cm,framed]

 \node[main node] (1) at (0,0){$Z$};
 \node[main node] (2) at (-2,0){$Z'$};
 \node[main node] (3) at (-4,0){$L$};
 \node[main node] (4) at (-4,-2){$L$};
 \node[main node] (5) at (-2,-2){$W$};
 \node[main node] (6) at (0,-2){$E$};
 \node[main node] (7) at (-2,-4){$B_V$};
 \node[main node] (8) at (0,-4){$B_V$};
 \node[main node] (01) at (0,2){$0$};
 \node[main node] (02) at (-2,2){$0$};
 \node[main node] (03) at (-6,0){$0$};
 \node[main node] (04) at (-6,-2){$0$};
 \node[main node] (05) at (-2,-6){$0$};
 \node[main node] (06) at (0,-6){$0$};
 \node[main node] (07) at (2,0){$0$};
 \node[main node] (08) at (2,-2){$0$};

\draw[->, thin]   (01)  to  node  {$$}    (1);
\draw[->, thin]   (02)  to  node  {$$}    (2);
\draw[->, thin]   (03)  to  node  {$$}    (3);
\draw[->, thin]   (04)  to  node  {$$}    (4);
\draw[->, thin]   (3)  to  node  {$$}    (2);
\draw[->, thin]   (2)  to  node  {$$}    (1);
\draw[->, thin]   (1)  to  node  {$$}    (07);
\draw[->, thin]   (4)  to  node  {$$}    (5);
\draw[->, thin]   (5)  to  node  {$$}    (6);
\draw[->, thin]   (6)  to  node  {$$}    (08);
\draw[->, thin]   (2)  to  node  {$$}    (5);
\draw[->, thin]   (5)  to  node  {$$}    (7);
\draw[->, thin]   (7)  to  node  {$$}    (05);
\draw[->, thin]   (1)  to  node  {$$}    (6);
\draw[->, thin]   (6)  to  node  {$$}    (8);
\draw[->, thin]   (8)  to  node  {$$}    (06);
\draw[-, double]   (3)  to  node  {$$}    (4);
\draw[-, double]   (7)  to  node  {$$}    (8);

 \end{tikzpicture}
\]%
\end{minipage}}\hfill{}%
\begin{minipage}[t]{0.6\columnwidth}%
\[
\suc[L][W][E][\,][h']
\]
with $W\in\omega$ and 
$L\in\mathcal{U}$. 
Hence, considering the
pull-back of $h$ and $h'$, we get an exact sequence 
\[
\eta':\quad \suc[L][Z'][Z][\,][\:]
\]
with $Z'\in\omega_{\V}^{\vee}$ and $L\in\U$. Note that $\eta'$ is the second of the two exact sequences we are looking for. Moreover,
observe that $\coresdimr{\mathcal{\omega}}{Z'}{\V}\leq n$.

It remains to prove that $\coresdimr{\omega}{B_{V}}{\V}=n-1$.
Since $\omega\subseteq\U$, we have%
\end{minipage} 
\[
\coresdimr{\U}{B_{V}}{\V}\leq\coresdimr{\omega}{B_{V}}{\V}\leq n-1\mbox{.}
\]
Moreover, by the exact sequence $\eta$, we have 
\[
n=\coresdimr{\U}Z{\V}\leq1+\coresdimr{\U}{B_{V}}{\V}\leq n\mbox{,}
\]
Therefore, $\coresdimr{\omega}{B_{V}}{\V}=n-1$.
\

(b)  From the last pull-back in the proof of (a), we have that $Z'\in\V$
if $Z\in\V$.
\

(c) Let $\omega\subseteq{}^{\bot}\U.$ Then, by   Lemma \ref{lem:dims vs res cores} we have $\omega^{\vee}\subseteq{}^{\bot}\U.$  Let $Z\in\U_{\V}^{\vee}$.
Consider the first exact sequence in (a). Observe that $C_{Z}\in{}^{\bot}\U$
and $M_{Z}\in\U\subseteq\left(\omega^{\vee}\right)^{\bot}$.\\
\begin{minipage}[t]{0.5\columnwidth}%
To show that $g_{Z}$ is an $\U$-preenvelope, observe that
any morphism $f:Z\rightarrow U$, with $U\in\U$, factors
through $g_{Z}$ if, and only if, the exact sequence induced by the
push-out of $f$ and $g_{Z}$ splits (which is the case since $C_{Z}\in{}^{\bot}\U$). The fact that $f_Z$ is  a $\omega^\vee$-precover is proved by dual arguments. 
\end{minipage}\hfill{}%
\fbox{\begin{minipage}[t]{0.45\columnwidth}%
\[
\begin{tikzpicture}[-,>=to,shorten >=1pt,auto,node distance=1.cm,main node/.style=,x=1.5cm,y=1.5cm]

   \node[main node] (KX) at (0,0)      {$C_Z$};
   \node[main node] (E) [left of=KX]  {$E$};
   \node[main node] (Y) [left of=E]  {$U$};
   \node[main node] (02) [left of=Y]  {$0$};
   \node[main node] (01) [right of=KX]  {$0$};

   \node[main node] (KX2) [above of=KX]  {$C_Z$};
   \node[main node] (MX) [left of=KX2]  {$M_Z$};
   \node[main node] (X) [left of=MX]  {$Z$};
   \node[main node] (04) [left of=X]  {$0$};
   \node[main node] (03) [right of=KX2]  {$0$};

\draw[->, thin]   (KX)  to node  {$$}  (01);
\draw[->, thin]   (E)  to node  {$$}  (KX);
\draw[->, thin]   (Y)  to node  {$$}  (E);
\draw[->, thin]   (02)  to node  {$$}  (Y);

\draw[->, thin]   (KX2)  to node  {$$}  (03);
\draw[->, thin]   (MX)  to  node  {$$}  (KX2);
\draw[->, thin]   (X)  to  node [above]  {$g_Z$}  (MX);
\draw[->, thin]   (04)  to  node   {$$}  (X);

\draw[-, double]   (KX2)  to  node  {$$}  (KX);
\draw[->, thin]   (MX)  to  node   {$$}  (E);
\draw[->, thin]   (X)  to  node [left]   {$f$}  (Y);
\draw[->, dashed]   (MX)  to  node   {$$}  (Y);
\draw[->, dashed]   (KX2)  to  node   {$$}  (E);
   
\end{tikzpicture}
\]%
\end{minipage}} \\
\end{proof}

Let us commence the study of the relations between the relative homological
dimensions and the relative (co)resolution dimensions. We begin with the intuitive
extension of some known results. It is worth mentioning that the following result
 is used in \cite{Argudin-Mendoza2} to study the properties of an $n$-$\X$-tilting class $\mathcal{T}.$ 
 
\begin{prop} \label{prop:coresdimr vs pdr-1} Let $\mathcal{C}$ be an abelian category,
$\mathcal{X},\mathcal{T}\subseteq\mathcal{C}$ and $\alpha\subseteq\mathcal{T}^{\bot}\cap\mathcal{X}^{\bot}$ be a
relative cogenerator in $\mathcal{X}.$ 
Then, the following inequalities hold true.
\begin{enumerate}
\item[$\mathrm{(a)}$]  $\coresdimr{\mathcal{T}^{\bot}\cap\mathcal{X}}{\mathcal{X}}{\mathcal{X}}\leq\pdr[\mathcal{X}][\mathcal{T}].$

\item[$\mathrm{(b)}$]  $\pdr[\mathcal{X}][{}{}^{\bot}\left(\mathcal{T}^{\bot}\right)]\leq\pdr[\mathcal{X}][{}{}^{\bot}\left(\mathcal{T}{}^{\bot}\cap\mathcal{X}\right)]\leq\pdr[\mathcal{X}][\mathcal{T}]$.
\end{enumerate}
\end{prop}
\begin{proof} (a)  We can assume that $n:=\pdr[\mathcal{X}][\mathcal{T}]<\infty$. If $n=0$
then $\mathcal{X}\subseteq\mathcal{T}^{\bot}$ and thus 
$\coresdimr{\mathcal{T}^{\bot}\cap\mathcal{X}}{\mathcal{X}}{\mathcal{X}}=0.$
\

Let $n\geq1$. Since $\alpha$ is a relative cogenerator in $\mathcal{X}$,
for every $A\in\mathcal{X},$ there is an exact sequence 
$
0\rightarrow A\rightarrow W_{0}\rightarrow\cdots\rightarrow W_{n-1}\rightarrow Q\rightarrow0\mbox{,}
$
with $Q\in\mathcal{T}^{\bot}\cap\mathcal{X}$ and $W_{i}\in\alpha$
$\forall i\in[0,n-1]$ (see Proposition \ref{prop:M ortodonal cerrado por n-cocientes sii pdM=00003Dn}).
Therefore, $\coresdimr{\mathcal{T}^{\bot}\cap\mathcal{X}}A{\mathcal{X}}\leq n$.
\

(b) Since $^{\bot}\left(\mathcal{T}^{\bot}\right)\subseteq{}^{\bot}(\mathcal{T}^{\bot}\cap\mathcal{X})$,
it is enough to prove that $\pdr[\mathcal{X}][^{\bot}(\mathcal{T}^{\bot}\cap\mathcal{X})]\leq\pdr[\mathcal{X}][\mathcal{T}]$.
Assume $n:=\pdr[\mathcal{X}][\mathcal{T}]<\infty$. Then, by (a), there is an exact sequence 
$0\rightarrow A\rightarrow W_{0}\rightarrow\cdots\rightarrow W_{n-1}\rightarrow W_n\rightarrow0,$ where  $W_i\in\mathcal{T}^{\bot}\cap\mathcal{X}$ $\forall i\in[0,n].$ Let $Y\in{}{}^{\bot}\left(\mathcal{T}^{\bot}\cap\mathcal{X}\right)$
and $A\in\mathcal{X}.$ Since $W_i\in\mathcal{T}^{\bot}\cap\mathcal{X},$ we get that $W_i\in Y^\perp$ $\forall\,i.$ Hence, by the shifting Lemma, 
$
0=\Extx[k][][Y][W_n]\cong\Extx[k+n][][Y][A]\quad\forall k>0\mbox{.}
$
Therefore, $\pdr[\mathcal{X}][Y]\leq n$.
\end{proof}

\begin{lem} \label{lem:(4.7/51homologia_relativa)} For 
any
 exact sequence  $\suc[A][B][C]$
in  $\mathcal{C}$ and  $\mathcal{X}\subseteq\mathcal{C},$ the following inequalities hold true.
\begin{enumerate}
\item $\idr[\mathcal{X}][B]\leq\max\left\{ \idr[\mathcal{X}][A],\idr[\mathcal{X}][C]\right\}.$

\item $\idr[\mathcal{X}][A]\leq\max\left\{ \idr[\mathcal{X}][B],\idr[\mathcal{X}][C]+1\right\}.$

\item $\idr[\mathcal{X}][C]\leq\max\left\{ \idr[\mathcal{X}][B],\idr[\mathcal{X}][A]-1\right\}.$
\end{enumerate}
\end{lem}
\begin{proof}
The proof is straightforward.
 \end{proof}

\begin{lem} \cite[p.16]{Auslander-Buchweitz}\label{lem:(4.1homologia_relativa)}
$\pdr[\mathcal{Y}][\mathcal{X}]=\idr[\mathcal{X}][\mathcal{Y}],$ for any $\mathcal{X},\mathcal{Y}\subseteq\mathcal{C}.$
\end{lem}

\begin{thm} \label{thm:-Para-}\cite[Thm. 2.1]{tiltstrat} 
$\idr[\mathcal{X}][L]\leq\idr[\mathcal{X}][\mathcal{Y}]+\coresdimr{\mathcal{Y}}{L}{}$, for all  $L\in\mathcal{C}$ and $\mathcal{X},\mathcal{Y}\subseteq\mathcal{C}.$
\end{thm}

The following result is a generalization of Theorem \ref{thm:-Para-}.

\begin{thm}\label{thm:(4.2homologia_relativa)}  Let $\mathcal{C}$ be an abelian
category and $\mathcal{X}$,$\mathcal{Y}$,$\mathcal{Z}\subseteq\mathcal{C}.$
Then, 
\[
\idr[\mathcal{X}][L]\leq\idr[\mathcal{X}][\mathcal{Y}]+\coresdimr{\mathcal{Y}}L{\mathcal{Z}}\quad\forall L\in\mathcal{C}\mbox{.}
\]
 Furthermore, if $\mathcal{Z}$ is closed under extensions, then 
\[
\idr[\mathcal{X}][L]\leq\idr[\mathcal{X}][\mathcal{Y}\cap\mathcal{Z}]+\coresdimr{\mathcal{Y}}L{\mathcal{Z}}\quad\forall L\in\mathcal{Z}.
\]
 \end{thm}
\begin{proof} 
Note that $\coresdimr{\mathcal{Y}}L{\,}\leq\coresdimr{\mathcal{Y}}L{\mathcal{Z}}$
$\forall L\in\mathcal{C}$. Hence, by Theorem  \ref{thm:-Para-}, we have that $\idr[\mathcal{X}][L]\leq\idr[\mathcal{X}][\mathcal{Y}]+\coresdimr{\mathcal{Y}}L{\mathcal{Z}}$
$\forall L\in\mathcal{C}.$ 

Let $\mathcal{Z}$ be closed under extensions and $L\in\Z.$ Assume that $\idr[\mathcal{X}][\mathcal{Y}\cap\mathcal{Z}]=n<\infty$
and $\coresdimr{\mathcal{Y}}L{\mathcal{Z}}=m<\infty$. We prove, 
by induction on $m,$ that $\idr[\mathcal{X}][L]\leq n+m$.\\
If $m=0$, we have $L\cong M\in\mathcal{Y}\cap\mathcal{Z}$ . Let
$m=1$. Since $\Z$ is closed under extensions, there is an exact sequence 
$
\suc[L][Y_{0}][Y_{1}]
$
with $Y_{0},Y_{1}\in\mathcal{Y}\cap\mathcal{Z}.$ Thus, for every
$X\in\mathcal{X}$, we have the exact sequence 
\[
\Extx[k-1][][X][Y_{1}]\rightarrow\Extx[k][][X][L]\rightarrow\Extx[k][][X][Y_{0}]\mbox{,}
\]
 where $\Extx[k-1][][X][Y_{1}]=0$ and $\Extx[k][][X][Y_{0}]=0,$ for
any $k>n+1$. Therefore, $\idr[\mathcal{X}][L]\leq n+1=\idr[\mathcal{X}][\mathcal{Y}\cap\mathcal{Z}]+\coresdimr{\mathcal{Y}}L{\mathcal{Z}}.$
\

Let $m\geq2$. Then, by using that $\Z$ is closed under extensions, we get an exact sequence $0\rightarrow L\stackrel{f}{\rightarrow}Y_{0}\rightarrow\cdots\rightarrow Y_{m}\rightarrow 0,$  where $Y_{i}\in\mathcal{Y}\cap\mathcal{Z}$ $\forall i\in[0,m]$,
and $\coresdimr{\mathcal{Y}}K{\mathcal{Z}}=m-1$ for $K:=\Cok[f]$.
Hence, by inductive hypothesis, 
\[
\idr[\mathcal{X}][K]\leq\idr[\mathcal{X}][\mathcal{Y}\cap\mathcal{Z}]+\coresdimr{\mathcal{Y}}K{\mathcal{Z}}=\idr[\mathcal{X}][\mathcal{Y}\cap\mathcal{Z}]+\coresdimr{\mathcal{Y}}L{\mathcal{Z}}-1.
\]
 Finally, by Lemma \ref{lem:(4.7/51homologia_relativa)} it follows that
\begin{alignat*}{1}
\idr[\mathcal{X}][L] & \leq\max\left\{ \idr[\mathcal{X}][Y_{0}],\idr[\mathcal{X}][K]+1\right\} \\
 & \leq\max\left\{ \idr[\mathcal{X}][Y_{0}],\idr[\mathcal{X}][\mathcal{Y}\cap\mathcal{Z}]+\coresdimr{\mathcal{Y}}L{\mathcal{Z}}\right\} \\
 & \leq\idr[\mathcal{X}][\mathcal{Y}\cap\mathcal{Z}]+\coresdimr{\mathcal{Y}}L{\mathcal{Z}}\mbox{.}
\end{alignat*}
\end{proof}

In the following lemma, we study the relative dimensions induced by
the classes in an $\mathcal{X}$-complete pair $(\A,\B).$

\begin{lem} \label{thm:(4.8homologia_relativa)}  For $\mathcal{X}\subseteq\mathcal{C}$ and a right $\mathcal{X}$-complete pair $\p$
  in $\C,$  the following statements
hold true. 
\begin{enumerate}
\item[$\mathrm{(a)}$] For any $M\in\C,$ we have
 $$\idr[\mathcal{X}][M]=\max\left\{ \idr[\mathcal{A}\cap\mathcal{X}][M],\idr[\mathcal{B}\cap\mathcal{X}][M]\right\} \leq\max\left\{ \idr[\mathcal{A}][M],\idr[\mathcal{B}][M]\right\}. $$

\item[$\mathrm{(b)}$] If $\mathcal{A}\cup\mathcal{B}\subseteq\mathcal{X}$, then $\idr[\mathcal{X}][M]=\max\left\{ \idr[\mathcal{A}][M],\idr[\mathcal{B}][M]\right\} \,\forall M\in\mathcal{C}\mbox{.}$

\item[$\mathrm{(c)}$] If $\p$ is $\mathcal{X}$-hereditary, then $\idr[\mathcal{X}][M]=\idr[\B\cap\X][M]\:\forall M\in\B\cap\mathcal{X}\mbox{.}$
\end{enumerate}
\end{lem}
\begin{proof}
(a)  Since $\p$ is right $\mathcal{X}$-complete,  for every $N\in\mathcal{X}$,
we have an exact sequence 
$
\suc[N][B][A]
$
with $B\in\mathcal{B}\cap\mathcal{X}$ and $A\in\mathcal{A}\cap\mathcal{X}$.
Let $M\in\mathcal{C}$. Then, by Lemma \ref{lem:(4.7/51homologia_relativa)} and Lemma
\ref{lem:(4.1homologia_relativa)}, it follows that 
\begin{alignat*}{1}
\idr[\{N\}][M]=\pdr[\{M\}][N] & \leq\max\left\{ \pdr[\{M\}][B],\pdr[\{M\}][A]-1\right\} \\
 & \leq\max\left\{ \pdr[\{M\}][\mathcal{B}\cap\mathcal{X}],\pdr[\{M\}][\mathcal{A}\cap\mathcal{X}]\right\} \\
 & \leq\max\left\{ \idr[\mathcal{B}\cap\mathcal{X}][M],\idr[\mathcal{A}\cap\mathcal{X}][M]\right\} \\
 & \leq\max\left\{ \idr[\mathcal{B}][M],\idr[\mathcal{A}][M]\right\} \mbox{.}
\end{alignat*}
Hence, $\idr[\mathcal{X}][M]\leq\max\left\{ \idr[\mathcal{B}\cap\mathcal{X}][M],\idr[\mathcal{A}\cap\mathcal{X}][M]\right\} \leq\max\left\{ \idr[\mathcal{B}][M],\idr[\mathcal{A}][M]\right\} \mbox{.}$
On the other hand, it is clear that $\idr[\mathcal{X}][M]\geq\max\left\{ \idr[\mathcal{B\cap\mathcal{X}}][M],\idr[\mathcal{A\cap\mathcal{X}}][M]\right\}.$
\

(b) It follows by (a).
\

(c) It follows from (a) since  $\idr[\mathcal{A\cap X}][\B\cap\X]=0.$ 
\end{proof}

\begin{prop}\label{thm:(4.8homologia_relativa)-1} 
Let  $\mathcal{X}\subseteq\mathcal{C}$ and $\p$
be a right $\mathcal{X}$-complete and $\mathcal{X}$-hereditary pair in $\C$
such that $\left(\mathcal{A}\cap\mathcal{X}\right)^{\bot}\cap\mathcal{A}\cap\mathcal{X}\subseteq\mathcal{B}\cap\mathcal{X}$
and $\mathcal{B}=\smdx[\mathcal{B}]$. Then, for any $X\in\mathcal{X},$
we have: 
\begin{enumerate}
\item[$\mathrm{(a)}$] $\coresdimr{\mathcal{B}}X{\,}\leq\idr[\mathcal{A}\cap\mathcal{X}][X]\leq\idr[\mathcal{A}\cap\mathcal{X}][\mathcal{B}]+\coresdimx{\mathcal{B}}X\mbox{;}$
\vspace{0.2cm}
\item[$\mathrm{(b)}$] $\coresdimr{\mathcal{B}\cap\mathcal{X}}X{\mathcal{X}}=\idr[\mathcal{A}\cap\mathcal{X}][X]=\coresdimx{\mathcal{B}\cap\mathcal{X}}X\leq\\
\leq\coresdimx{\mathcal{B}\cap\mathcal{X}}{\mathcal{A}\cap\mathcal{X}}+1;$
\vspace{0.2cm}
\item[$\mathrm{(c)}$] $\coresdimr{\mathcal{B}}X{\mathcal{X}}=\coresdimr{\mathcal{B}\cap\mathcal{X}}X{\mathcal{X}}$
if $\mathcal{X}$ is closed under extensions;
\vspace{0.2cm}

\item[$\mathrm{(d)}$] $\coresdimr{\mathcal{B}}X{\,}\leq\idr[\mathcal{A}][X]\leq\idr[\mathcal{A}][\mathcal{B}]+\coresdimx{\mathcal{B}}{\mathcal{A}}+1\mbox{;}$
\vspace{0.2cm}

\item[$\mathrm{(e)}$] $(\mathcal{A}\cap\mathcal{X})^{\bot}\cap\mathcal{X}\subseteq\mathcal{B}$.
\end{enumerate}
\end{prop}
\begin{proof}
Let $X\in \mathcal{X}$. We prove firstly that
\begin{equation}
\coresdimr{\mathcal{B}\cap\mathcal{X}}X{\mathcal{X}}\leq\idr[\mathcal{A}\cap\mathcal{X}][X]\mbox{.}\label{eq:1-2}
\end{equation}

We may assume that $n:=\idr[\mathcal{A}\cap\mathcal{X}][X]<\infty.$ Let
$n=0$. In order to show that $\coresdimr{\mathcal{B}\cap\mathcal{X}}X{\mathcal{X}}=0,$
it is enough to prove that $X\in\mathcal{B}$. Since $\p$ is right
$\mathcal{X}$-complete, there is an exact sequence 
\[
\eta:\:\suc[X][B][A]\mbox{, with }B\in\mathcal{B}\cap\mathcal{X}\mbox{ and }A\in\mathcal{A}\cap\mathcal{X}.
\]
 Observe that $\eta$ splits since $n=0$. Hence $X$ is a direct
summand of $B$ and thus $X\in\mathcal{B}$. Note that this argument gives a proof of  (e).

Let $n\geq1$. Since $\p$ is right $\mathcal{X}$-complete, there
is an exact sequence 
\[
\epsilon:\;0\rightarrow X\stackrel{g_{0}}{\rightarrow}B_{0}\stackrel{g_{1}}{\rightarrow}B_{1}\stackrel{g_{2}}{\rightarrow}...\stackrel{g_{n-1}}{\rightarrow}B_{n-1}\rightarrow A_{n}\rightarrow0
\]
with $B_{i}\in\mathcal{B}\cap\mathcal{X}$ and $A_{i+1}:=\Cok[g_{i}]\in\mathcal{A}\cap\mathcal{X}$
$\forall i\in[0,n-1]$. Now, since $\idr[\mathcal{A}\cap\mathcal{X}][\mathcal{B}\cap\mathcal{X}]=0$
and $n=\idr[\mathcal{A}\cap\mathcal{X}][X]$, we have 
\[
\Extx[i][][A][A_{n}]\cong\Extx[n+i][][A][X]=0\:\forall A\in\mathcal{A}\cap\mathcal{X}\,\forall i\geq1.
\]
 Hence $A_{n}\in\left(\mathcal{A}\cap\mathcal{X}\right)^{\bot}\cap\mathcal{A}\cap\mathcal{X}\subseteq\mathcal{B}\cap\mathcal{X}$
and thus, from $\epsilon$, we have $\coresdimr{\mathcal{B}}X{\mathcal{X}}\leq n;$
proving (\ref{eq:1-2}).

Note that (a) follows from (\ref{eq:1-2}) and Theorem \ref{thm:-Para-}. On
the other hand, by using that $\idr[\mathcal{A}\cap\mathcal{X}][\mathcal{B}\cap\mathcal{X}]=0$,
(\ref{eq:1-2}) and Theorem \ref{thm:-Para-}, we get 
\begin{alignat}{1}
\coresdimr{\mathcal{B}}X{\,} & \leq\coresdimr{\mathcal{B}}X{\mathcal{X}}\nonumber \\
 & \leq\coresdimr{\mathcal{B}\cap\mathcal{X}}X{\mathcal{X}}\nonumber \\
 & \leq\idr[\mathcal{A}\cap\mathcal{X}][X]\nonumber \\
 & \leq\idr[\mathcal{A}\cap\mathcal{X}][\mathcal{B}\cap\mathcal{X}]+\coresdimx{\mathcal{B}\cap\mathcal{X}}X\nonumber \\
 & =\coresdimx{\mathcal{B}\cap\mathcal{X}}X\nonumber \\
 & \leq\coresdimr{\mathcal{B}\cap\mathcal{X}}X{\mathcal{X}}\mbox{.}\label{eq:2-2}
\end{alignat}
 Note that (\ref{eq:2-2}) proves the first two equalities of (b). Next, let us prove that 
\[
\idr[\mathcal{A}\cap\mathcal{X}][X]\leq\coresdimr{\mathcal{B}\cap\mathcal{X}}{\mathcal{A}\cap\mathcal{X}}{\,}+1\mbox{, }\forall X\in\mathcal{X}\mbox{.}
\]
Consider $X\in\mathcal{X}$. Since $\p$ is right $\mathcal{X}$-complete,
there is an exact sequence 
\[
\suc[X][B][A]\mbox{, with }B\in\mathcal{B}\cap\mathcal{X}\mbox{ and }A\in\mathcal{A}\cap\mathcal{X}.
\]
 Then, by Lemma \ref{lem:(4.7/51homologia_relativa)} and Theorem \ref{thm:-Para-},
we have 
\begin{alignat*}{1}
\idr[\mathcal{A}\cap\mathcal{X}][X] & \leq\max\left\{ \idr[\mathcal{A}\cap\mathcal{X}][B],\idr[\mathcal{A}\cap\mathcal{X}][A]+1\right\} \\
 & =\idr[\mathcal{A}\cap\mathcal{X}][A]+1\\
 & \leq\idr[\mathcal{A}\cap\mathcal{X}][\mathcal{A}\cap\mathcal{X}]+1\\
 & \leq\idr[\mathcal{A}\cap\mathcal{X}][\mathcal{B}\cap\mathcal{X}]+\coresdimx{\mathcal{B}\cap\mathcal{X}}{\mathcal{A}\cap\mathcal{X}}+1\\
 & =\coresdimx{\mathcal{B}\cap\mathcal{X}}{\mathcal{A}\cap\mathcal{X}}+1\mbox{.}
\end{alignat*}
Therefore, $\idr[\mathcal{A}\cap\mathcal{X}][X]\leq\coresdimx{\mathcal{B}\cap\mathcal{X}}{\mathcal{A}\cap\mathcal{X}}+1$
$\forall X\in\mathcal{X}$.

To prove (c), we assume that $\mathcal{X}$ is closed under extensions. Then,
by Theorem \ref{thm:(4.2homologia_relativa)}, we have 
\[
\idr[\mathcal{A}\cap\mathcal{X}][X]\leq\idr[\mathcal{A}\cap\mathcal{X}][\mathcal{B}\cap\mathcal{X}]+\coresdimr{\mathcal{B}}X{\mathcal{X}}=\coresdimr{\mathcal{B}}X{\mathcal{X}}\,\forall X\in\mathcal{X}\mbox{,}
\]
and by (\ref{eq:2-2}), it follows (c).

Finally, since $\coresdimx{\mathcal{B}}X\leq\idr[\mathcal{A}\cap\mathcal{X}][X]\leq\idr[\mathcal{A}][X]$
$\forall X\in\mathcal{X}$ by (\ref{eq:2-2}), the proof of (d) is
similar to the proof of (b). Indeed, since $\p$ is right $\mathcal{X}$-complete,
for any $X\in\mathcal{X}$, there is an exact sequence 
\[
\suc[X][B][A]\mbox{, with }B\in\mathcal{B}\cap\mathcal{X}\mbox{ and }A\in\mathcal{A}\cap\mathcal{X}.
\]
Then, by Lemma \ref{lem:(4.7/51homologia_relativa)} and Theorem \ref{thm:-Para-},
we have 
\begin{alignat*}{1}
\idr[\mathcal{A}][X] & \leq\max\left\{ \idr[\mathcal{A}][B],\idr[\mathcal{A}][A]+1\right\} \\
 & \leq\max\left\{ \idr[\mathcal{A}][\mathcal{B}],\idr[\mathcal{A}][\mathcal{A}]+1\right\} \\
 & \leq\max\left\{ \idr[\mathcal{A}][\mathcal{B}],\idr[\mathcal{A}][\mathcal{B}]+\coresdimx{\mathcal{B}}{\mathcal{A}}+1\right\} \\
 & =\idr[\mathcal{A}][\mathcal{B}]+\coresdimx{\mathcal{B}}{\mathcal{A}}+1\mbox{.}
\end{alignat*}
\end{proof}

\begin{cor} \label{cor:(4.9homologia_relativa)} For a closed under extensions class $\mathcal{X}\subseteq\mathcal{C}$ and a right $\mathcal{X}$-complete pair $\p$ in $\C$ which is
$\mathcal{X}$-hereditary and $\mathcal{B}=\smdx[\mathcal{B}],$ the following statements
hold true.
\begin{enumerate}
\item[$\mathrm{(a)}$] $\idr[\mathcal{A}\cap\mathcal{X}][M]=\coresdimr{\mathcal{B}}M{\mathcal{X}}$
$\forall M\in(\mathcal{X},\mathcal{B})^{\vee}.$

\item[$\mathrm{(b)}$] Let $\left(\mathcal{A}\cap\mathcal{X}\right)^{\bot}\cap\mathcal{A}\cap\mathcal{X}\subseteq\mathcal{B}\cap\mathcal{X}.$ Then, for any  $X\in\mathcal{X},$ we have
\begin{alignat*}{1}
\coresdimr{\mathcal{B}}X{\mathcal{X}} & =\coresdimr{\mathcal{B}\cap\mathcal{X}}X{\mathcal{X}}\\
 & =\idr[\mathcal{A}\cap\mathcal{X}][X]\\
 & =\coresdimr{\mathcal{B}\cap\mathcal{X}}X{\,}\\
 & \leq\coresdimr{\mathcal{B\cap\mathcal{X}}}{\mathcal{A}\cap\mathcal{X}}{\mbox{ }}+1.
\end{alignat*}
\end{enumerate}
\end{cor}
\begin{proof} (a)  Let $M\in\mathcal{B}_{\mathcal{X}}^{\vee}\cap\mathcal{X}$. We proceed
by induction on $n:=\coresdimr{\mathcal{B}}M{\mathcal{X}}$.\\
If $n=0$ then $M\cong N\in\mathcal{B}\cap\mathcal{X}$. Thus $\idr[\mathcal{A}\cap\mathcal{X}][M]=0$
since
 $\idr[\mathcal{A}\cap\mathcal{X}][\mathcal{B}\cap\mathcal{X}]=0$.
\

We assert that  $\idr[\mathcal{A}\cap\mathcal{X}][M]\geq1$ if $n>0.$
Indeed, since $\p$ is right $\mathcal{X}$-complete and $M\in\mathcal{X}$,
there is an exact sequence 
\[
\suc[M][B][A]\mbox{, with }B\in\mathcal{B}\cap\mathcal{X}\mbox{ and }A\in\mathcal{A}\cap\mathcal{X}.
\]
Note that this sequence does not split since $M\notin\mathcal{B}\cap\mathcal{X}$
 ($\coresdimr{\mathcal{B}}M{\mathcal{X}}=n>0$). Hence $\Extx[1][][A][M]\neq0$
and thus $\idr[\mathcal{A}\cap\mathcal{X}][M]\geq1;$ proving our assertion.
\

Let $n=1$. Since $\X$ is closed under extensions, there is an exact sequence 
\[
\suc[M][Y_{0}][Y_{1}]\mbox{, with }Y_{0},Y_{1}\in\mathcal{B}\cap\mathcal{X}\mbox{.}
\]
By Lemma \ref{lem:(4.7/51homologia_relativa)}, $\idr[\mathcal{A}\cap\mathcal{X}][M]\leq\max\left\{ \idr[\mathcal{A}\cap\mathcal{X}][Y_{0}],\idr[\mathcal{A}\cap\mathcal{X}][Y_{1}]+1\right\} =1\mbox{.}$
Therefore, $\idr[\mathcal{A}\cap\mathcal{X}][M]=1$ since $\idr[\mathcal{A}\cap\mathcal{X}][M]\geq1$.\\
Let $n>1$. By inductive hypothesis, $\idr[\mathcal{A}\cap\mathcal{X}][N]=\coresdimr{\mathcal{B}}N{\mathcal{X}}$
for any $N\in\mathcal{X}$ with $\coresdimr{\mathcal{B}}N{\mathcal{X}}\leq n-1$.
Since $\mathcal{X}$ is closed under extensions, we have an exact sequence
\[
0\rightarrow M\rightarrow Y_{0}\overset{f_{1}}{\rightarrow}Y_{1}\rightarrow\cdots\overset{f_{n}}{\rightarrow}Y_{n}\rightarrow0\mbox{,}
\]
 with $Y_{n}\in\mathcal{B}\cap\mathcal{X}$, $Y_{i}\in\mathcal{\mathcal{B}}\cap\mathcal{X}$,
$\im{f_{i}}\in\mathcal{X}$ $\forall i\in[0,n-1]$. Moreover, for
$K:=\im{f_{1}}$, we have $\coresdimr{\mathcal{B}}K{\mathcal{X}}=n-1$.
Consider the exact sequence 
\[
\suc[M][Y_{0}][K]\mbox{.}
\]
Note that $\idr[\mathcal{A}\cap\mathcal{X}][K]=\coresdimr{\mathcal{B}}K{\mathcal{X}}=n-1$
by inductive hypothesis. Hence, by Lemma \ref{lem:(4.7/51homologia_relativa)},
we have 
\begin{alignat*}{1}
\idr[\mathcal{A\cap\mathcal{X}}][M] & \leq\max\left\{ \idr[\mathcal{A\cap\mathcal{X}}][Y_{0}],\idr[\mathcal{A}\cap\mathcal{X}][K]+1\right\} =n\mbox{,}\quad\mbox{ and }\\
n-1 & =\idr[\mathcal{A}\cap\mathcal{X}][K]\leq\max\left\{ \idr[\mathcal{A\cap\mathcal{X}}][Y_{0}],\idr[\mathcal{A}\cap\mathcal{X}][M]-1\right\} =\idr[\mathcal{A}\cap\mathcal{X}][M]-1
\end{alignat*}
since $\idr[\mathcal{A\cap\mathcal{X}}][M]\geq1$. Therefore, $\idr[\mathcal{A}\cap\mathcal{X}][M]=n.$
\

(b) It follows from Proposition \ref{thm:(4.8homologia_relativa)-1}(c,b).
\end{proof}

Before going any further in the study of $\mathcal{X}$-complete pairs,
we recall from \cite{ABsurvey}  the notion of thick class. We will be interested in reviewing
the 
main
 properties of this kind of classes with regard to their
generators and cogenerators. 

\begin{defn}
\label{def: thick gruesa}\cite[Def. 2.2]{ABsurvey} Let $\mathcal{C}$
be an abelian category and $\mathcal{X}\subseteq\mathcal{C}.$ 
\begin{enumerate}
\item[$\mathrm{(a)}$] $\mathcal{X}$ is \textbf{right thick} if it is closed
under extensions, direct summands and mono-cokernels.
\item[$\mathrm{(b)}$] $\mathcal{X}$ is \textbf{left thick} if it is closed
under extensions, direct summands and epi-kernels.
\item[$\mathrm{(c)}$] $\mathcal{X}$ is  \textbf{thick} if it is right  and left
thick.
\end{enumerate}
\end{defn}

\begin{lem} \label{lem:lemitia auslander reiten} Let $\omega\subseteq\X\subseteq\C$ with $\X$ closed under epi-kernels.  Then $\omega^{\vee}\subseteq\mathcal{X}$.
\end{lem}

\begin{proof} 
The 
proof is straightforward and we left it to the reader.
\end{proof}

The next Lemma is a generalization of \cite[Prop. 2.7]{ABsurvey}.
The only difference between them is a subtle precision on the kind of the resolutions
that are used. 

\begin{lem} \label{lem:(5.1homologia_relativa)} For the classes $\omega,\X\subseteq\mathcal{C}$ such that $\omega$ is $\mathcal{X}$-injective, the following statements hold true. 
\begin{enumerate}
\item[$\mathrm{(a)}$] $\omega^{\wedge}$ is $\mathcal{X}$-injective.
\item[$\mathrm{(b)}$] If $\omega=\smdx[\omega]$ is a relative cogenerator in $\mathcal{X},$
then 
\[
\mathcal{X}\cap\mathcal{X}^{\bot}=\mathcal{X}\cap\omega^{\wedge}=(\omega,\mathcal{X})^{\wedge}=\omega\mbox{.}
\]
\end{enumerate}
\end{lem}

\begin{proof} 
Item 
(a) follows from the dual of  Lemma \ref{lem:(4.10homologia_relativa)}. Let us show (b). Indeed, from (a),  it is clear that $(\omega,\mathcal{X})^{\wedge}\subseteq\mathcal{X}\cap\omega^{\wedge}\subseteq\mathcal{X}\cap\mathcal{X}^{\bot}$.
\

Let $X\in\mathcal{X}\cap\mathcal{X}^{\bot}$. Since $\omega$
is a relative cogenerator in $\mathcal{X}$, there is an exact sequence $\eta:\;\suc[X][W][X'],$  with $W\in\omega$  and $X'\in\mathcal{X}.$
Moreover, the exact sequence $\eta$  splits since $X\in\mathcal{X}^{\bot}$ and thus $X\in\omega$. Therefore, 
\[
(\omega,\mathcal{X})^{\wedge}\mathcal{\subseteq X}\cap\omega^{\wedge}\subseteq\mathcal{X}\cap\mathcal{X}^{\bot}\subseteq\omega\mbox{.}
\]
Finally  $\omega\subseteq(\omega,\mathcal{X})^{\wedge}$ since $\omega\subseteq\mathcal{X}$ and $0\in\omega.$
\end{proof}

The following lemma is a compilation of results from the Auslander-Buchweitz theory, see \cite[Prop. 4.2, Lem. 4.3 and Prop. 4.7]{Auslander-Buchweitz}. 

\begin{lem} \label{lem:(5.3homologia_relativa)}
For the classes $\omega,\X\subseteq\mathcal{C}$ such that
 $\omega=\smdx[\omega]$ is an $\mathcal{X}$-injective relative cogenerator
in $\mathcal{X},$ the following statements hold true.
\begin{enumerate}
\item[$\mathrm{(a)}$] $\mathcal{X}\cap\omega^{\vee}=\left\{ X\in\mathcal{X}\,|\:\idr[\mathcal{X}][X]<\infty\right\}.$
\item[$\mathrm{(b)}$] $\idr[\mathcal{X}][M]=\coresdimx{\omega}M$ $\forall M\in\mathcal{X}\cap\omega^{\vee}.$
\item[$\mathrm{(c)}$] If $\mathcal{X}$ is left thick, then $\omega^{\vee}$ is left thick
and 
\[
\omega^{\vee}=\left\{ X\in\mathcal{X}\,|\:\idr[\mathcal{X}][X]<\infty\right\} \mbox{.}
\]
\end{enumerate}
\end{lem}

The following result is a generalization of the previous lemma and will play an important role in the proof of Theorem \ref{thm:el reemplazo }. 

\begin{lem}\label{lem:gen aus id vs cores con cogenerador} For the classes
$\p[\mathcal{W}][\nu]\subseteq\mathcal{C}^{2}$
and $\mathcal{X}\subseteq\mathcal{C}$  such that $\nu=\smdx[\nu]$
is a $\mathcal{W}\cap\mathcal{X}$-injective relative cogenerator
in $\mathcal{W}\cap\mathcal{X},$ the following statements hold true.
\begin{enumerate}
\item[$\mathrm{(a)}$] $\mathcal{W}\cap\mathcal{X}\cap\nu{}_{\mathcal{X}}^{\vee}=\left\{ W\in\mathcal{W}\cap\mathcal{X}\,|\:\idr[\mathcal{W}\cap\mathcal{X}][W]<\infty\right\}.$
\item[$\mathrm{(b)}$] $\idr[\mathcal{W}\cap\mathcal{X}][M]=\coresdimr{\nu}M{\mathcal{X}}$
$\forall M\in\mathcal{W}\cap\mathcal{X}\cap\nu_{\mathcal{X}}^{\vee}.$
\item[$\mathrm{(c)}$] If $\mathcal{W}\cap\mathcal{X}$ is left thick, then $\nu_{\mathcal{X}}^{\vee}$
is left thick and 
\[
\nu_{\mathcal{X}}^{\vee}=\left\{ W\in\mathcal{W}\cap\mathcal{X}\,|\:\idr[\mathcal{W}\cap\mathcal{X}][W]<\infty\right\} \mbox{.}
\]
\end{enumerate}
\end{lem}
\begin{proof}
Let $M\in\mathcal{W}\cap\mathcal{X}\cap\nu_{\mathcal{X}}^{\vee}$.
By Theorem \ref{thm:(4.2homologia_relativa)}, we have 
\[
\idr[\mathcal{W}\cap\mathcal{X}][M]\leq\idr[\mathcal{W}\cap\mathcal{X}][\nu]+\coresdimr{\nu}M{\mathcal{X}}=\coresdimr{\nu}M{\mathcal{X}}<\infty\mbox{.}
\]
For every $M\in\mathcal{W}\cap\mathcal{X}$ with $\idr[\mathcal{W}\cap\mathcal{X}][M]=n<\infty$,
we have $\coresdimr{\nu}M{\mathcal{X}}\leq n$. Indeed, since $\nu$
is a relative cogenerator in $\mathcal{W}\cap\mathcal{X}$, there
is an exact sequence 
\[
\eta:\:\suc[M][N_{0}][Z][\,][\,]\mbox{, with }N_{0}\in\nu\mbox{ and }Z\in\mathcal{W}\cap\mathcal{X}.
\]
 Now, if $n=0$ then $\eta$ splits and thus $\coresdimr{\nu}M{\mathcal{X}}=0$.
Let $n\geq1$. We can build an exact sequence 
\[
0\rightarrow M\rightarrow N_{0}\overset{f_{0}}{\rightarrow}N_{1}\rightarrow\cdots\rightarrow N_{n-2}\overset{f_{n-2}}{\rightarrow}N_{n-1}\rightarrow Z\rightarrow0\mbox{,}
\]
with $N_{i}\in\nu\;\forall i\in[0,n-1]$, $\im{f_{i}}\in\mathcal{W}\cap\mathcal{X}\;\forall i\in[0,n-2]$,
and $Z\in\mathcal{W}\cap\mathcal{X}$. Hence, by the Shifting Lemma,
it follows that 
\[
\Extx[k][][W][Z]=\Extx[n+k][][W][M]=0\quad\forall\, W\in\mathcal{W}\cap\mathcal{X},\;\forall k>0\mbox{.}
\]
 Therefore, by  Lemma \ref{lem:(5.1homologia_relativa)}(b),  $Z\in\mathcal{W}\cap\mathcal{X}\cap(\mathcal{W}\cap\mathcal{X})^{\bot}=\nu$ and thus
 $\coresdimr{\nu}M{\mathcal{X}}\leq n,$
proving (a) and (b).

Let $\mathcal{W}\cap\mathcal{X}$ be left thick. Then,  by Lemma \ref{lem:lemitia auslander reiten}
we have $\nu^{\vee}\subseteq\mathcal{W}\cap\mathcal{X}$. Therefore, the equality of (c) follows from (a) since 
$\nu_{\mathcal{X}}^{\vee}\subseteq\nu^{\vee}\subseteq\mathcal{W}\cap\mathcal{X}$. Lastly, the left thickness of
 $\nu_X^\vee$ follows from Lemma \ref{lem:(4.7/51homologia_relativa)}. 
\end{proof}

In what follows,  we will be interested in the study of the closure
properties of the classes $(\mathcal{X},\mathcal{Y})_{\infty}^{\vee}$, $(\mathcal{X},\mathcal{Y})^{\vee}$
and $(\mathcal{X},\mathcal{Y})_{n}^{\vee}$. In particular, we will get sufficient conditions for this classes to
be thick. 

\begin{prop}\label{prop:auslander reiten generalizado} For an abelian category $\mathcal{C}$ and $\mathcal{X}$,$\mathcal{Y}\subseteq\mathcal{C},$ the following statements hold true.
\begin{enumerate}
\item[$\mathrm{(a)}$] Let $\mathcal{Y}=\Y^{\oplus_{<\infty}},$  $\mathcal{X}$
be closed under extensions and $(\mathcal{X},\mathcal{Y})_{\infty}^{\vee}\subseteq{}{}^{\bot_{1}}\mathcal{Y}$.
Then, for a given exact sequence $\suc[A][B][C][\,][\,]\mbox{,}$ with
$A,C\in(\mathcal{X},\mathcal{Y})_{\infty}^{\vee}$, we have that
$
\coresdimr{\mathcal{Y}}B{\mathcal{X}}\leq\max\left\{ \coresdimr{\mathcal{Y}}A{\mathcal{X}},\coresdimr{\mathcal{Y}}C{\mathcal{X}}\right\} \mbox{.}
$
Furthermore, $(\mathcal{X},\mathcal{Y})_{\infty}^{\vee}$ and $(\mathcal{X},\mathcal{Y})^{\vee}$
are closed under extensions. 
\item[$\mathrm{(b)}$] Let $\mathcal{X}=\smdx[\mathcal{X}]$ and $(\mathcal{X},\mathcal{Y})_{\infty}^{\vee}$
be both closed under extensions. Then $(\mathcal{X},\mathcal{Y})_{\infty}^{\vee}$
is closed under direct summands.
\end{enumerate}
\end{prop}
\begin{proof} (a) Let $\eta_{0}:\;\suc[A][B][C][u][v]$ be an exact sequence with $A,C\in(\mathcal{X},\mathcal{Y})_{\infty}^{\vee}.$ Then
 $B\in\mathcal{X}$ since $\mathcal{X}$ is closed under extensions.
On the other hand, by definition, there are exact sequences 
\[
\eta_{A}^{0}:\;\suc[A][Y_{A}][A_{1}][a][\,]\:\mbox{ and }\:\eta_{C}^{0}:\;\suc[C][Y_{C}][C_{1}][c][\,]\mbox{,}
\]
with $Y_{A},Y_{C}\in\mathcal{Y}$ and $A_{1},C_{1}\in(\mathcal{X},\mathcal{Y})_{\infty}^{\vee}$.
Since $C\in(\mathcal{X},\mathcal{Y})_{\infty}^{\vee}\subseteq{}{}^{\bot_{1}}\mathcal{Y}\subseteq{}{}^{\bot_{1}}Y_{A}\mbox{,}$
we have the exact sequence 
\[
0\rightarrow\Homx[\mathcal{C}][C][Y_{A}]\rightarrow\Homx[\mathcal{C}][B][Y_{A}]\rightarrow\Homx[\mathcal{C}][A][Y_{A}]\rightarrow0\mbox{.}
\]
Therefore, there is a morphism $\alpha:B\rightarrow Y_{A}$ such that
$\alpha u=a$. Consider the morphism $b:=\left(\begin{smallmatrix}\alpha\\
cv
\end{smallmatrix}\right):B\rightarrow Y_{A}\oplus Y_{C}\mbox{.}$ Since 
\[
\left(\begin{smallmatrix}1\\
0
\end{smallmatrix}\right)a=\left(\begin{smallmatrix}a\\
0
\end{smallmatrix}\right)=\left(\begin{smallmatrix}\alpha u\\
0
\end{smallmatrix}\right)=\left(\begin{smallmatrix}\alpha\\
cv
\end{smallmatrix}\right)u\quad\mbox{ and }\quad\left(\begin{smallmatrix}0 & 1\end{smallmatrix}\right)\left(\begin{smallmatrix}\alpha\\
cv
\end{smallmatrix}\right)=cv\mbox{,}
\]
by the Snake Lemma, we get the exact sequences 
\[
\eta_{B}^{0}:\;\suc[B][Y_{0}][B_{1}][b][\,]\;\mbox{ and }\;\eta_{1}:\;\suc[A_{1}][B_{1}][C_{1}][\,][\,]\mbox{,}
\]
where $A_{1}$,$C_{1}\in(\mathcal{X},\mathcal{Y})_{\infty}^{\vee}$,
$Y_{0}:=Y_{A}\oplus Y_{C}\in\mathcal{Y}$ and $B_{1}\in\mathcal{X}$.
In order to repeat the argument recursively, assume we have exact
sequences 
\[
\eta_{B}^{k-1}:\;\suc[B_{k-1}][Y_{k-1}][B_{k}][b_{k-1}][\,]\;\mbox{ and }\;\eta_{k}:\;\suc[A_{k}][B_{k}][C_{k}][u_{k}][v_{k}]\mbox{,}
\]
where $A_{k},$ $C_{k}\in(\mathcal{X},\mathcal{Y})_{\infty}^{\vee}$,
$X_{k}\in\mathcal{Y}$ and $B_{k}\in\mathcal{X}$ $\forall\, k\leq n$.
Observe that there are exact sequences
\[
\eta_{A}^{n}:\;\suc[A_{n}][Y_{A,n}][A_{n+1}][a_{n}][\,]\:\mbox{ and }\:\eta_{C}^{n}:\;\suc[C_{n}][Y_{C,n}][C_{n+1}][c_{n}][\,]\mbox{,}
\]
with $Y_{A,n},Y_{C,n}\in\mathcal{Y}$ and $A_{n+1},C_{n+1}\in(\mathcal{X},\mathcal{Y})_{\infty}^{\vee}$.
Since $C_{n}\in(\mathcal{X},\mathcal{Y})_{\infty}^{\vee}\subseteq{}{}^{\bot_{1}}\mathcal{Y}\mbox{,}$
we have the exact sequence 
\[
0\rightarrow\Homx[\mathcal{C}][C_{n}][Y_{A,n}]\rightarrow\Homx[\mathcal{C}][B_{n}][Y_{A,n}]\rightarrow\Homx[\mathcal{C}][A_{n}][Y_{A,n}]\rightarrow0\mbox{.}
\]
Therefore, there is a morphism $\alpha_{n}:B_{n}\rightarrow Y_{A,n}$
such that $\alpha_{n}u_{n}=a_{n}$. Consider the morphism $b_{n}:=\left(\begin{smallmatrix}\alpha_{n}\\
c_{n}v_{n}
\end{smallmatrix}\right):B_{n}\rightarrow Y_{A,n}\oplus Y_{C,n}\mbox{.}$ Since 
\[
\left(\begin{smallmatrix}1\\
0
\end{smallmatrix}\right)a_{n}=\left(\begin{smallmatrix}a_{n}\\
0
\end{smallmatrix}\right)=\left(\begin{smallmatrix}\alpha_{n}u_{n}\\
0
\end{smallmatrix}\right)=\left(\begin{smallmatrix}\alpha_{n}\\
c_{n}v_{n}
\end{smallmatrix}\right)u_{n}\quad\mbox{ and }\quad\left(\begin{smallmatrix}0 & 1\end{smallmatrix}\right)\left(\begin{smallmatrix}\alpha_{n}\\
c_{n}v_{n}
\end{smallmatrix}\right)=c_{n}v_{n}\mbox{,}
\]
 by the Snake Lemma, we get the exact sequences 
\[
\eta_{B}^{n}:\;\suc[B_{n}][Y_{n}][B_{n+1}][b_{n}][\,]\;\mbox{ and }\;\eta_{n+1}:\;\suc[A_{n+1}][B_{n+1}][C_{n+1}][\,][\,]\mbox{,}
\]
where $A_{n+1}$,$C_{n+1}\in(\mathcal{X},\mathcal{Y})_{\infty}^{\vee}$,
$Y_{n}:=Y_{A,n}\oplus Y_{C,n}\in\mathcal{Y}$ and $B_{n+1}\in\mathcal{X}$.
Observe that the family of short exact sequences $\left\{ \eta_{B}^{i}\right\} _{i=0}^{\infty}$ induces a long exact sequence from where we get that
 $B\in(\mathcal{X},\mathcal{Y})_{\infty}^{\vee}$. Furthermore,  if $A,C\in(\mathcal{X},\mathcal{Y})^{\vee}$ then the families of exact 
sequences $\left\{ \eta_{A}^{k}\right\} _{k=0}^{\infty}$ and $\left\{ \eta_{B}^{k}\right\} _{k=0}^{\infty}$
can be chosen in a way that they form a $(\mathcal{X},\mathcal{Y})$-coresolution
of minimal length. Hence, for 
\[
m:=\max\left\{ \coresdimr{\mathcal{Y}}A{\mathcal{X}},\coresdimr{\mathcal{Y}}C{\mathcal{X}}\right\} \mbox{,}
\]
we have $A_{k}=0=C_{k}\:\forall k>m$. Thus, by considering the family of exact
sequences $\left\{ \eta_{k}\right\} _{k=1}^{\infty}$, we have $B_{k}=0\,\forall k>m$.
Therefore $\coresdimr{\mathcal{Y}}B{\mathcal{X}}\leq m$.
\vspace{0.2cm}

(b) Consider a split exact sequence 
\[
\suc[W][V][U][\,][f]\mbox{, with \ensuremath{V}\ensuremath{\in}(\ensuremath{\mathcal{X}},\ensuremath{\mathcal{Y})_{\infty}^{\vee}} . }
\]
Then $U,W\in\mathcal{X}$ since $\mathcal{X}=\smdx[\mathcal{X}].$  Let us show that $W\in(\mathcal{X},\mathcal{Y})_{\infty}^{\vee}.$ Indeed,  since $V\in(\mathcal{X},\mathcal{Y})_{\infty}^{\vee},$
there is an exact sequence 
\[
\suc[V][Y_{0}][V_{1}][g]\mbox{, with \ensuremath{Y_{0}}\ensuremath{\in\mathcal{Y}} and \ensuremath{V_{1}}\ensuremath{\in}(\ensuremath{\mathcal{X}},\ensuremath{\mathcal{Y})_{\infty}^{\vee}} .}
\]
 Now, by considering the push-out of $f$ with $g$, we get the exact
sequences\\
\begin{minipage}[t]{0.55\columnwidth}%
\begin{alignat*}{1}
\eta:\; & \suc[U][W_{1}][V_{1}][\,][\,]\mbox{,}\\
\mu_{0}:\; & \suc[W][Y_{0}][W_{1}][\,][\,]\mbox{.}
\end{alignat*}
Since $U,V_{1}\in\mathcal{X}$, from $\eta$ we get that $W_{1}\in\mathcal{X}$.
Moreover, by making the coproduct of $\eta$ with the short exact sequence
\[
\suc[W][W][0][1][\,]\mbox{,}
\]
we get the short exact sequence 
\[
\suc[V][W\oplus W_{1}][V_{1}]\mbox{.}
\]
\end{minipage}\hfill{}%
\fbox{\begin{minipage}[t]{0.4\columnwidth}%
\[
\begin{tikzpicture}[-,>=to,shorten >=1pt,auto,node distance=1cm,main node/.style=,x=.45cm,y=.45cm]

 \node[main node] (1) at (0,0){$U$};
 \node[main node] (2) at (-2,0){$V$};
 \node[main node] (3) at (-4,0){$W$};
 \node[main node] (4) at (-4,-2){$W$};
 \node[main node] (5) at (-2,-2){$Y_0$};
 \node[main node] (6) at (0,-2){$W_1$};
 \node[main node] (7) at (-2,-4){$V_1$};
 \node[main node] (8) at (0,-4){$V_1$};
 \node[main node] (01) at (0,2){$0$};
 \node[main node] (02) at (-2,2){$0$};
 \node[main node] (03) at (-6,0){$0$};
 \node[main node] (04) at (-6,-2){$0$};
 \node[main node] (05) at (-2,-6){$0$};
 \node[main node] (06) at (0,-6){$0$};
 \node[main node] (07) at (2,0){$0$};
 \node[main node] (08) at (2,-2){$0$};

\draw[->, thin]   (01)  to  node  {$$}    (1);
\draw[->, thin]   (02)  to  node  {$$}    (2);
\draw[->, thin]   (03)  to  node  {$$}    (3);
\draw[->, thin]   (04)  to  node  {$$}    (4);

\draw[->, thin]   (3)  to  node  {$$}    (2);
\draw[->, thin]   (2)  to  node  {$$}    (1);
\draw[->, thin]   (1)  to  node  {$$}    (07);
\draw[->, thin]   (4)  to  node  {$$}    (5);
\draw[->, thin]   (5)  to  node  {$$}    (6);
\draw[->, thin]   (6)  to  node  {$$}    (08);
\draw[->, thin]   (2)  to  node  {$$}    (5);
\draw[->, thin]   (5)  to  node  {$$}    (7);
\draw[->, thin]   (7)  to  node  {$$}    (05);
\draw[->, thin]   (1)  to  node  {$$}    (6);
\draw[->, thin]   (6)  to  node  {$$}    (8);
\draw[->, thin]   (8)  to  node  {$$}    (06);
\draw[-, double]   (3)  to  node  {$$}    (4);
\draw[-, double]   (7)  to  node  {$$}    (8);
   
\end{tikzpicture}
\]%
\end{minipage}}\\
\vspace{0.2cm}

Observe that $V,V_{1}\in(\mathcal{X},\mathcal{Y})_{\infty}^{\vee}$.
Hence $W\oplus W_{1}\in(\mathcal{X},\mathcal{Y})_{\infty}^{\vee}$
since $(\mathcal{X},\mathcal{Y})_{\infty}^{\vee}$ is closed under extensions.
Then, by repeating the above argument, we get a family of
exact sequences 
$
\left\{ \mu_{i}:\;\suc[W_{i}][Y_{i}][W_{i+1}][\,][\,]\right\} _{i=0}^{\infty},
$
where $W_{0}:=W$, $Y_{i}\in\mathcal{Y}$ and $W_{i}\in\mathcal{X}\:\forall\, i\geq0$.
Therefore $W\in(\mathcal{X},\mathcal{Y})_{\infty}^{\vee}$.
\end{proof}

The following theorem is inspired by  \cite[Thm. 3.29]{relgor},
which in turn is a generalization of \cite[Prop. 5.1]{auslandereiten}. The closure properties of the relative classes, involved in Theorem \ref{thm:teo nuevo}, Theorem \ref{thm:el reemplazo } and Corollary \ref{cor:coronuevo}, play an important role in the study and development of $n$-$\X$-tilting theory in \cite{Argudin-Mendoza2}.

\begin{thm} \label{thm:teo nuevo}  Let $(\X,\Y)\subseteq\mathcal{C}^2$ be classes such that
$\Y=\Y^{\oplus_{<\infty}},$ $\mathcal{X}=\smdx[\mathcal{X}]$
is closed under extensions and $\Extx[1][\mathcal{C}][\mathcal{X}][\mathcal{X}\cap\mathcal{Y}]=0.$
Then, the following statements hold true.
\begin{enumerate}
\item[$\mathrm{(a)}$]  $\p[\mathcal{X}][\mathcal{Y}]_{\infty}^{\vee}=\p[\mathcal{X}][\mathcal{X}\cap\mathcal{Y}]_{\infty}^{\vee}$
and it is closed under extensions and direct summands. Moreover $(\mathcal{X},\mathcal{Y})^{\vee}=(\mathcal{X},\mathcal{X}\cap\mathcal{Y})^{\vee}$
and it is closed under extensions.
\item[$\mathrm{(b)}$] $\p[\mathcal{X}][\mathcal{Y}]_{\infty}^{\vee}$ is left thick if $\mathcal{X}$
is left thick.
\end{enumerate}
\end{thm}
\begin{proof} (a)  The equality $\p[\mathcal{X}][\mathcal{Y}]_{\infty}^{\vee}=\p[\mathcal{X}][\mathcal{X}\cap\mathcal{Y}]_{\infty}^{\vee}$
follows from the fact that $\mathcal{X}$ is closed under extensions.
 Since $\Extx[1][\mathcal{C}][\mathcal{X}][\mathcal{X}\cap\mathcal{Y}]=0$,
we have that  
$
(\mathcal{X},\mathcal{X}\cap\mathcal{Y})_{\infty}^{\vee}\subseteq\mathcal{X}\subseteq{}^{\bot_{1}}(\mathcal{X}\cap\mathcal{Y})\mbox{.}
$
Therefore, we get (a) by applying Proposition \ref{prop:auslander reiten generalizado}
to the pair $(\mathcal{X},\mathcal{X}\cap\mathcal{Y}).$
\

(b) Let $\mathcal{X}$ be closed under epi-kernels. Then, by (a), it is enough
to show that $\p[\mathcal{X}][\mathcal{X}\cap\mathcal{Y}]_{\infty}^{\vee}$
 is closed under epi-kernels. Consider an exact sequence 
$\suc[A][B][C][a][\,],$  with $B,C\in\p[\mathcal{X}][\mathcal{X}\cap\mathcal{Y}]_{\infty}^{\vee}.$
In particular, there is an exact sequence\\
\begin{minipage}[t]{0.55\columnwidth}%
\[
\suc[B][W_{0}][C_{0}][b][\,]\mbox{,}
\]
with $W_{0}\in\mathcal{X}\cap\mathcal{Y}$ and $C_{0}\in(\mathcal{X},\mathcal{X}\cap\mathcal{Y})_{\infty}^{\vee}$.
By the composition $ba:A\to W_0$ and the Snake Lemma, we get the exact sequences
\begin{alignat*}{1}
\eta:\:\suc[A][W_{0}][C'][\,][\,] & \mbox{,}\\
\eta':\:\suc[C][C'][C_{0}][\,][\,] & \mbox{.}
\end{alignat*}
 Since $C,C_{0}\in(\mathcal{X},\mathcal{X}\cap\mathcal{Y})_{\infty}^{\vee}$, it follows from (a) that
 $C'\in(\mathcal{X},\mathcal{X}\cap\mathcal{Y})_{\infty}^{\vee},$
and since $\mathcal{X}$ is left thick, we have $A\in\mathcal{X}.$ 
Therefore, $\eta$ shows that  $A\in(\mathcal{X},\mathcal{X}\cap\mathcal{Y})_{\infty}^{\vee}$.%
\end{minipage}\hfill{}%
\fbox{\begin{minipage}[t]{0.4\columnwidth}%
\[
\begin{tikzpicture}[-,>=to,shorten >=1pt,auto,node distance=1cm,main node/.style=,x=.45cm,y=.45cm]

 \node[main node] (1) at (0,0){$C$};
 \node[main node] (2) at (-2,0){$B$};
 \node[main node] (3) at (-4,0){$A$};
 \node[main node] (4) at (-4,-2){$A$};
 \node[main node] (5) at (-2,-2){$W_0$};
 \node[main node] (6) at (0,-2){$C'$};
 \node[main node] (7) at (-2,-4){$C_0$};
 \node[main node] (8) at (0,-4){$C_0$};
 \node[main node] (01) at (0,2){$0$};
 \node[main node] (02) at (-2,2){$0$};
 \node[main node] (03) at (-6,0){$0$};
 \node[main node] (04) at (-6,-2){$0$};
 \node[main node] (05) at (-2,-6){$0$};
 \node[main node] (06) at (0,-6){$0$};
 \node[main node] (07) at (2,0){$0$};
 \node[main node] (08) at (2,-2){$0$};

\draw[->, thin]   (01)  to  node  {$$}    (1);
\draw[->, thin]   (02)  to  node  {$$}    (2);
\draw[->, thin]   (03)  to  node  {$$}    (3);
\draw[->, thin]   (04)  to  node  {$$}    (4);

\draw[->, thin]   (3)  to  node  {$$}    (2);
\draw[->, thin]   (2)  to  node  {$$}    (1);
\draw[->, thin]   (1)  to  node  {$$}    (07);
\draw[->, thin]   (4)  to  node  {$$}    (5);
\draw[->, thin]   (5)  to  node  {$$}    (6);
\draw[->, thin]   (6)  to  node  {$$}    (08);
\draw[->, thin]   (2)  to  node  {$$}    (5);
\draw[->, thin]   (5)  to  node  {$$}    (7);
\draw[->, thin]   (7)  to  node  {$$}    (05);
\draw[->, thin]   (1)  to  node  {$$}    (6);
\draw[->, thin]   (6)  to  node  {$$}    (8);
\draw[->, thin]   (8)  to  node  {$$}    (06);
\draw[-, double]   (3)  to  node  {$$}    (4);
\draw[-, double]   (7)  to  node  {$$}    (8);
   
\end{tikzpicture}
\]%
\end{minipage}}\\
\end{proof}

\begin{thm}\label{thm:el reemplazo } Let $\p[\mathcal{Z}][\nu]\subseteq\mathcal{C}^{2}$ be classes such that $\mathcal{Z}=\smdx[\mathcal{Z}]$
is closed under extensions, $\addx[\nu]=\nu$ and $\nu$ is $\mathcal{Z}$-injective.
Then, the following statements hold true.
\begin{enumerate}
\item[$\mathrm{(a)}$] The classes $\p[\mathcal{Z}][\nu]_{\infty}^{\vee}$ and 
$\p[\mathcal{Z}][\nu]^{\vee}$
are closed under extensions and direct summands. Moreover, we have the equalities
  \begin{enumerate}
  \item[$\mathrm{(a1)}$] $\p[\mathcal{Z}][\nu]_{\infty}^{\vee}=\p[\mathcal{Z}][\mathcal{Z}\cap\nu]_{\infty}^{\vee},$
  \vspace{0.2cm}
  \item[$\mathrm{(a2)}$] $\p[\mathcal{Z}][\nu]^{\vee}=\p[\mathcal{Z}][\mathcal{Z}\cap\nu]^{\vee}=\left\{ M\in(\mathcal{Z},\nu)_{\infty}^{\vee}\,|\:\idr[(\mathcal{Z},\nu)_{\infty}^{\vee}][M]<\infty\right\}.$
  \end{enumerate}
  \vspace{0.2cm}
\item[$\mathrm{(b)}$] $\idr[(\mathcal{Z},\nu)_{\infty}^{\vee}][M]=\coresdimr{\nu}M{\mathcal{Z}}\;\forall\, M\in(\mathcal{Z},\nu)^{\vee}$.
\vspace{0.2cm}
\item[$\mathrm{(c)}$] The classes $(\mathcal{Z},\nu)_{\infty}^{\vee}$ and $(\mathcal{Z},\nu)^{\vee}$
are left thick if $\mathcal{Z}$ is left thick.
\end{enumerate}
\end{thm}

\begin{proof} (a) \& (b) By Theorem \ref{thm:teo nuevo} we get: the first two equalities
in (a), $(\mathcal{Z},\nu)_{\infty}^{\vee}$ is closed under extensions
and direct summands, and $(\mathcal{Z},\nu)^{\vee}$ is closed
under extensions.
\

Let $\mathcal{W}=(\mathcal{Z},\nu)_{\infty}^{\vee}.$ Since
$\nu$ is $\mathcal{Z}$-injective, it follows that $\nu$ is a $\mathcal{W}$-injective
relative cogenerator in $\mathcal{W}$. Hence, by Lemma \ref{lem:gen aus id vs cores con cogenerador} (a)
\[
\p[\mathcal{Z}][\nu]^{\vee}=\mathcal{W}\cap\p[\mathcal{Z}][\nu]^{\vee}=\mathcal{W}\cap\mathcal{Z}\cap\nu_{\mathcal{Z}}^{\vee}=\left\{ M\in\mathcal{W}\cap\mathcal{Z}\,|\:\idr[\mathcal{W}\cap\mathcal{Z}][M]<\infty\right\} \mbox{.}
\]
Note that $\mathcal{W}\cap\mathcal{Z}=\mathcal{W}$. Consequently,
\begin{equation*}
(*)\quad{\small(\mathcal{Z},\nu)^{\vee}=\left\{ M\in(\mathcal{Z},\nu)_{\infty}^{\vee}\,|\:\idr[(\mathcal{Z},\nu)_{\infty}^{\vee}][M]<\infty\right\} =\left\{ M\in\mathcal{W}\,|\:\idr[\mathcal{W}][M]<\infty\right\}} \mbox{.}
\end{equation*}
 Furthermore, we have (b) from Lemma \ref{lem:gen aus id vs cores con cogenerador}(b).
 \
 
Let us show that $(\mathcal{Z},\nu)^{\vee}$ is closed under direct summands. Indeed, let $M\in(\mathcal{Z},\nu)^{\vee}$ and $M=M_{1}\oplus M_{2}$. Since
$(\mathcal{Z},\nu)^{\vee}\subseteq(\mathcal{Z},\nu)_{\infty}^{\vee}$
and $(\mathcal{Z},\nu)_{\infty}^{\vee}$ is closed under direct summands,
we have $M_{1},M_{2}\in(\mathcal{Z},\nu)_{\infty}^{\vee}$. Now, by using $(*)$ and that
$$\max\left\{ \idr[\mathcal{W}][M_{1}],\idr[\mathcal{W}][M_{2}]\right\} =\idr[\mathcal{W}][M]<\infty,$$
it follows that $M_{1},M_{2}\in(\mathcal{Z},\nu)^{\vee}$.
\

(c) Let $\mathcal{Z}$ be left thick. By Theorem \ref{thm:teo nuevo} (b),
we have that $(\mathcal{Z},\nu)_{\infty}^{\vee}$ is left thick. Now,
from (a) and the proof of Theorem \ref{thm:teo nuevo} (b), it follows that 
$(\mathcal{Z},\nu)^{\vee}$
is left thick.
\end{proof}

\begin{cor}\label{cor:coronuevo} Let 
$\mathcal{X}\subseteq\mathcal{C}$ be left thick and $\mathcal{T}=\addx[\mathcal{T}]\subseteq\mathcal{C}$
be such that $\mathcal{T}\subseteq\mathcal{T}^{\bot}\cap\mathcal{X}.$
Then, the following statements hold true.
\begin{enumerate}
\item[$\mathrm{(a)}$] The class $\mathcal{Q}:=({}^{\bot}\mathcal{T}\cap\mathcal{X},\mathcal{T})_{\infty}^{\vee}$
is left thick.

\item[$\mathrm{(b)}$] $\mathcal{T}_{\mathcal{X}}^{\vee}=\mathcal{T}^{\vee}=\{M\in\mathcal{Q}\;|\;\idr[\mathcal{Q}][M]<\infty\}$
and it is left thick. 
\end{enumerate}
\end{cor}

\begin{proof}
Note that $\mathcal{Q}\subseteq{}^{\bot}\mathcal{T}\cap\mathcal{X},$
 $\mathcal{Q}\cap\mathcal{X}=\mathcal{Q}$ and $\mathcal{T}\subseteq{}^{\bot}\mathcal{T}\cap\mathcal{Q}.$  
Using that $\Extx[i][][^{\bot}\mathcal{T}][\mathcal{T}]=0$
$\forall i>0$, it follows that $\mathcal{T}$ is 
$({}^{\bot}\mathcal{T}\cap\mathcal{X})$-injective
and thus $\T$ is $\mathcal{Q}$-injective since $\mathcal{Q}\subseteq{}^{\bot}\mathcal{T}\cap\mathcal{X}.$ 
Lastly, observe that ${}^{\bot}\mathcal{T}\cap\mathcal{X}$ is left thick since $\X$ and ${}^{\bot}\mathcal{T}$ are left thick.
Now, by applying Theorem \ref{thm:el reemplazo }
to the pair $({}^{\bot}\mathcal{T}\cap\mathcal{X},\mathcal{T}),$
we get that $\mathcal{Q}$ is left thick; proving (a). Finally, by applying
Lemma \ref{lem:(5.3homologia_relativa)} and Lemma \ref{lem:gen aus id vs cores con cogenerador}
to the pair $(\mathcal{Q},\mathcal{T})$ and the class $\mathcal{X},$ it can be seen that
(b) holds true.
\end{proof}

We can now return to the study of the $\mathcal{X}$-complete pairs in an abelian category $\C.$
We will be focusing on deepening our understanding of the relations between
the different induced relative dimensions. Furthermore, we will see that, under
certain hypotheses, for an $\mathcal{X}$-complete pair $\p[\mathcal{A}][\mathcal{B}]$ in $\C,$
the class $\mathcal{A}\cap\mathcal{B}\cap\mathcal{X}$ is a relative
generator in $\mathcal{A}\cap\mathcal{X}$ and a relative cogenerator
in $\mathcal{B}\cap\mathcal{X}.$ Let us start by recalling the following
result proved by M. Auslander and R. O. Buchweitz in \cite{Auslander-Buchweitz}.

\begin{pro}\cite[Prop. 2.1]{Auslander-Buchweitz} \label{thm:(5.6homologia_relativa)}
Let $\p[\mathcal{X}][\omega]\subseteq\mathcal{C}^{2}$ be classes which are closed under direct summands, $\mathcal{X}$ be closed under extensions and let $\omega$ be an $\mathcal{X}$-injective
relative cogenerator in $\mathcal{X}.$ Then 
\begin{center}
$\pdr[\omega^{\wedge}][C]=\pdr[\omega][C]=\resdimx{\mathcal{X}}C\quad\forall\, C\in\mathcal{X}^{\wedge}.$
\end{center}
\end{pro}

 In the case of an $\mathcal{X}$-complete
and $\mathcal{X}$-hereditary pair, the above result can be strengthened, see Proposition \ref{pro:(5.7homologia_relativa)}. As we will see in \cite{Argudin-Mendoza2}, for an $n$-$\X$-tilting subcategory of $\C,$ the pair $({}^\perp(\T^\perp), \T^\perp)$ is always $\X$-complete and $\X$-hereditary. Therefore, Proposition \ref{pro:(5.7homologia_relativa)} and Theorem \ref{thm:(5.8homologia_relativa)}  will play an important role in the development of the $n$-$\X$-tilting theory  in \cite{Argudin-Mendoza2}.

\begin{pro}\label{pro:(5.7homologia_relativa)} For a class $\mathcal{X}\subseteq\mathcal{C},$ an $\mathcal{X}$-complete
and $\mathcal{X}$-hereditary pair $\p$ in $\C$ such that $\mathcal{A},$ $\mathcal{X}$
and $\mathcal{B}$ are closed under extensions and direct summands,  and $\omega:=\mathcal{A}\cap\mathcal{B}\cap\mathcal{X},$
the following statements hold true.
\begin{enumerate}
\item[$\mathrm{(a)}$] The class $\omega$ is an $\mathcal{A}\cap\mathcal{X}$-injective relative cogenerator
in $\mathcal{A}\cap\mathcal{X}.$

\item[$\mathrm{(b)}$] The class $\omega$ is an $\B\cap\X$-projective relative generator
in $\B\cap\X.$

\item[$\mathrm{(c)}$] For any $M\in\left(\mathcal{A},\mathcal{X}\right)^{\wedge},$ we have the equalities
\begin{center}
$\pdr[\mathcal{B}\cap\mathcal{X}][M]=\pdr[\omega][M]=\pdr[\omega^{\wedge}][M]=\resdimr{\mathcal{A}}M{\mathcal{X}}=\resdimx{\mathcal{X}\cap\mathcal{A}}M.$
\end{center}

\item[$\mathrm{(d)}$] $\pdr[\mathcal{B}\cap\mathcal{X}][M]=\resdimr{\omega}M{\mathcal{X}}\;$
$\forall\, M\in(\omega,\mathcal{X})^{\wedge}.$
\vspace{0.2cm}

\item[$\mathrm{(e)}$] For any $M\in\left(\X,\B\right)^{\vee},$ we have the equalities
\begin{center}
$\idr[\A\cap\mathcal{X}][M]=\idr[\omega][M]=\idr[\omega^{\vee}][M]=\coresdimr{\B}M{\mathcal{X}}=\coresdimx{\mathcal{X}\cap\B}M.$
\end{center}

\item[$\mathrm{(f)}$] $\idr[\A\cap\mathcal{X}][M]=\coresdimr{\omega}M{\mathcal{X}}\;$
$\forall\, M\in(\X,\omega)^{\vee}.$
\end{enumerate}
\end{pro}

\begin{proof} Note that (b) is the dual of (a), (e) is the dual of (c) and (f) is the dual of (d).
\

(a) We have that $\idr[\mathcal{A}\cap\mathcal{X}][\mathcal{B}\cap\mathcal{X}]=0$ since $\p$ is $\mathcal{X}$-hereditary.
In particular $\idr[\mathcal{A}\cap\mathcal{X}][\omega]\leq\idr[\mathcal{A}\cap\mathcal{X}][\mathcal{B}\cap\mathcal{X}]=0$ and thus 
 $\omega$ is $\mathcal{A}\cap\mathcal{X}$-injective.
 \
 
Let us show that $\omega$ is a relative cogenerator in $\mathcal{A}\cap\mathcal{X}.$ Indeed, since $\p$ is right $\mathcal{X}$-complete, for every $X\in\mathcal{A}\cap\mathcal{X},$
there is an exact sequence 
\[
\suc[X][W][X']\mbox{, with }W\in\mathcal{B}\cap\mathcal{X}\mbox{ and }X'\in\mathcal{A}\cap\mathcal{X}.
\]
Furthermore, $W\in\mathcal{A}\cap\mathcal{X}$ since $\mathcal{A}\cap\mathcal{X}$ is closed under extensions. Hence $W$ belongs to $\mathcal{A}\cap\mathcal{B}\cap\mathcal{X}=\omega,$ proving (a). 
\

(c) Observe, firstly, that $\left(\mathcal{A},\mathcal{X}\right)^{\wedge}\subseteq\left(\mathcal{A}\cap\mathcal{X}\right)^{\wedge}$ since $\mathcal{X}$ is
 closed under extensions.
\

Let $M\in\left(\mathcal{A},\mathcal{X}\right)^{\wedge}.$ Then by (a) and Proposition \ref{thm:(5.6homologia_relativa)},  $\pdr[\omega^{\wedge}][M]=\pdr[\omega][M]=\resdimx{\mathcal{A}\cap\mathcal{X}}M\mbox{.}$
By the dual of Corollary \ref{cor:(4.9homologia_relativa)} (a), 
$\pdr[\mathcal{B}\cap\mathcal{X}][M]=\resdimr{\mathcal{A}}M{\mathcal{X}}\mbox{.}$
Moreover $\resdimx{\mathcal{A}\cap\mathcal{X}}M\leq\resdimr{\mathcal{A}}M{\mathcal{X}}$ because $\mathcal{X}$ is closed under extensions.
Since we have that
\[
\pdr[\omega][M]=\pdr[\omega^{\wedge}][M]=\resdimx{\mathcal{A}\cap\mathcal{X}}M\leq\resdimr{\mathcal{A}}M{\mathcal{X}}=\pdr[\mathcal{B}\cap\mathcal{X}][M]\mbox{,}
\]
it is enough to show that $\pdr[\mathcal{B}\cap\mathcal{X}][M]\leq\pdr[\omega][M].$ In order to prove that, we can assume that $\pdr[\omega][M]=m<\infty$.
Then
\[
\pdr[\omega][M]\leq\pdr[\mathcal{B}\cap\mathcal{X}][M]=\resdimr{\mathcal{A}}M{\mathcal{X}}<\infty\mbox{.}
\]
and there is some $t\geq0$ such that $\pdr[\mathcal{B\cap\mathcal{X}}][M]=m+t$.
Let $B\in\mathcal{B}\cap\mathcal{X}$. By (b), we know that $\omega$
is a $\mathcal{B}\cap\mathcal{X}$-projective relative generator
in $\mathcal{B}\cap\mathcal{X}$. Hence, there is an exact sequence
$
0\rightarrow B_{t}\rightarrow A_{t-1}\rightarrow\cdots\rightarrow A_{0}\rightarrow B\rightarrow0\mbox{,}
$
with $B_{t}\in\mathcal{B}\cap\mathcal{X}$ and $A_{i}\in\omega\:\forall i\in[0,t-1].$
Since 
$
A_{i}\in\omega\subseteq M^{\bot_{>m}}\;\forall\, i\in[0,t-1]\mbox{,}
$
 by the Shifting Lemma we have 
$
\Extx[k][][M][B]\cong\Extx[k+t][][M][B_{t}]\,\forall k>m\mbox{.}
$
Now, by using that $\pdr[\mathcal{B\cap\mathcal{X}}][M]=m+t$, it follows
that 
\[
\Extx[k][][M][B]\cong\Extx[k+t][][M][B_{t}]=0\quad\forall\, k>m\mbox{.}
\]
 Therefore $\pdr[\mathcal{B}\cap\mathcal{X}][M]\leq m=\pdr[\omega][M]$.
\

(d) Let $M\in(\omega,\mathcal{X})^{\wedge}$. By (b), we have that
$\pdr[\mathcal{B}\cap\mathcal{X}][\omega]=0$. Then, by the dual 
of Theorem \ref{thm:(4.2homologia_relativa)}, we get that 
\[
\pdr[\mathcal{B}\cap\mathcal{X}][M]\leq\pdr[\mathcal{B}\cap\mathcal{X}][\omega]+\resdimr{\omega}M{\mathcal{X}}=\resdimr{\omega}M{\mathcal{X}}\mbox{.}
\]
We shall prove, by induction on $n=\resdimr{\omega}M{\mathcal{X}}$,
that $\pdr[\mathcal{B}\cap\mathcal{X}][M]=\resdimr{\omega}M{\mathcal{X}}$.
\\
If $n=0$ then $M\in\omega$ and thus $\pdr[\mathcal{B}\cap\mathcal{X}][M]=0.$
\

Let $n>0$. By inductive hypothesis $\pdr[\mathcal{B}\cap\mathcal{X}][N]=\resdimr{\omega}N{\mathcal{X}},$
for every $N\in\mathcal{X}$ with $\resdimr{\omega}N{\mathcal{X}}<n$.
Since $n=\resdimr{\omega}M{\mathcal{X}}$, we have an exact sequence
\[
\eta:\quad\suc[K][W_{0}][M]\mbox{, with }W_{0}\in\omega,\,K\in(\omega,\mathcal{X})^{\wedge}\mbox{,}
\]
and $\resdimr{\omega}K{\mathcal{X}}=n-1.$ Thus $\pdr[\mathcal{B}\cap\mathcal{X}][K]=n-1$.
Besides, by Lemma \ref{lem:(4.7/51homologia_relativa)}, 
\begin{equation}
n-1=\pdr[\mathcal{B}\cap\mathcal{X}][K]\leq\max\{\pdr[\mathcal{B}\cap\mathcal{X}][M]-1,\pdr[\mathcal{B}\cap\mathcal{X}][W_{0}]\}\mbox{.}\label{eq:maxmax}
\end{equation}
We assert that $\pdr[\mathcal{B}\cap\mathcal{X}][M]>0.$ Suppose that
 $\pdr[\mathcal{B}\cap\mathcal{X}][M]=0$. Since $\omega\subseteq\mathcal{A}$,
we have $(\omega,\mathcal{X})^{\wedge}\subseteq(\mathcal{A},\mathcal{X})^{\wedge}$.
Hence, by (c), $\pdr[\mathcal{B}\cap\mathcal{X}][M]=\pdr[\omega^{\wedge}][M]$.
Thus $\eta$ splits since $K\in\omega^{\wedge}.$ Consequently
$M\in\omega$, contradicting that $\resdimr{\omega}M{\mathcal{X}}>0;$ and the assertion follows. Then, by (\ref{eq:maxmax}), we get that $\pdr[\mathcal{B}\cap\mathcal{X}][M]=n$ since 
 $\pdr[\mathcal{B}\cap\mathcal{X}][M]>0$ and $\pdr[\mathcal{B}\cap\mathcal{X}][W_{0}]=0.$
\end{proof}

\begin{thm}\label{thm:(5.8homologia_relativa)} For a class $\mathcal{X}\subseteq\mathcal{C},$ an $\mathcal{X}$-complete
and $\mathcal{X}$-hereditary pair $\p$ in $\C$ such that $\mathcal{A},$ $\mathcal{X}$
and $\mathcal{B}$ are closed under extensions and direct summands,  and $\omega:=\mathcal{A}\cap\mathcal{B}\cap\mathcal{X},$
the following statements hold true.
\begin{itemize}
\item[$\mathrm{(a)}$] $\omega=\left(\mathcal{A}\cap\mathcal{X}\right)^{\bot}\cap\mathcal{A}\cap\mathcal{X}=\mathcal{A}\cap\mathcal{X}\cap\omega^{\wedge}=(\omega,\mathcal{A}\cap\mathcal{X})^{\wedge}.$ Furthermore, 
  \begin{itemize}
  \item[$\mathrm{(a1)}$] we have that 
  \begin{align*}
\pdr[\mathcal{X}][\mathcal{A}\cap\mathcal{X}] & =\pdr[\mathcal{A\cap\mathcal{X}}][\mathcal{A}\cap\mathcal{X}] & =\coresdimx{\mathcal{B}\cap\mathcal{X}}{\mathcal{A}\cap\mathcal{X}}\\
 & =\coresdimr{\mathcal{B}}{\mathcal{X}}{\mathcal{X}} & =\coresdimr{\mathcal{B}}{\mathcal{A}\cap\mathcal{X}}{\mathcal{X}}\\
 & =\coresdimx{\omega}{\mathcal{A}\cap\mathcal{X}} & =\coresdimx{\mathcal{B}\cap\mathcal{X}}{\mathcal{A}\cap\mathcal{X}}\\
 & =\coresdimr{\mathcal{B}\cap\mathcal{X}}{\mathcal{A}\cap\mathcal{X}}{\mathcal{X}} & =\coresdimr{\mathcal{B}\cap\mathcal{X}}{\mathcal{X}}{\mathcal{X}}\mbox{;}
\end{align*}
  
  \item[$\mathrm{(a2)}$] $\pdr[\mathcal{A}\cap\mathcal{X}][\mathcal{A}\cap\mathcal{X}]<\infty$
if, and only if, $\mathcal{A}\cap\mathcal{X}\subseteq\omega^{\vee}$
and $\pdr[\mathcal{X}][\omega]<\infty.$ Moreover, for $\pdr[\mathcal{A}\cap\mathcal{X}][\mathcal{A}\cap\mathcal{X}]<\infty,$  we have that
 \begin{center}
$\mathcal{X}\subseteq(\mathcal{X},\mathcal{B})^{\vee}\subseteq\left(\mathcal{B}\cap\mathcal{X}\right)^{\vee}\mbox{ and }\pdr[\mathcal{X}][\mathcal{A}\cap\mathcal{X}]=\pdr[\mathcal{X}][\omega].$
  \end{center}
  \end{itemize}

\item[$\mathrm{(b)}$]  $\omega={}^\perp\left(\B\cap\mathcal{X}\right)\cap\B\cap\mathcal{X}=\B\cap\mathcal{X}\cap\omega^{\vee}=(\B\cap\mathcal{X},\omega)^{\vee}.$ Furthermore, 
  \begin{itemize}
  \item[$\mathrm{(b1)}$] we have that 
  \begin{align*}
\idr[\mathcal{X}][\B\cap\mathcal{X}] & =\idr[\mathcal{\B\cap\mathcal{X}}][\B\cap\mathcal{X}] & =\resdimx{\A\cap\mathcal{X}}{\B\cap\mathcal{X}}\\
 & =\resdimr{\A}{\mathcal{X}}{\mathcal{X}} & =\resdimr{\A}{\B\cap\mathcal{X}}{\mathcal{X}}\\
 & =\resdimx{\omega}{\B\cap\mathcal{X}} & =\resdimx{\A\cap\mathcal{X}}{\B\cap\mathcal{X}}\\
 & =\resdimr{\A\cap\mathcal{X}}{\B\cap\mathcal{X}}{\mathcal{X}} & =\resdimr{\A\cap\mathcal{X}}{\mathcal{X}}{\mathcal{X}}\mbox{;}
\end{align*}
  
  \item[$\mathrm{(b2)}$] $\idr[\B\cap\mathcal{X}][\B\cap\mathcal{X}]<\infty$
if, and only if, $\B\cap\mathcal{X}\subseteq\omega^{\wedge}$
and $\idr[\mathcal{X}][\omega]<\infty.$ Moreover, for $\idr[\B\cap\mathcal{X}][\B\cap\mathcal{X}]<\infty,$  we have that
 \begin{center}
$\mathcal{X}\subseteq(\A,\mathcal{X})^{\wedge}\subseteq\left(\A\cap\mathcal{X}\right)^{\wedge}\mbox{ and }\idr[\mathcal{X}][\B\cap\mathcal{X}]=\idr[\mathcal{X}][\omega].$
  \end{center}
  \end{itemize}
\end{itemize}
\end{thm}
\begin{proof} Note firstly that (b) is the dual of (a). Thus, we need to prove (a).
\

By Proposition \ref{pro:(5.7homologia_relativa)} (a), $\omega$ is an 
$(\mathcal{A}\cap\mathcal{X})$-injective
relative cogenerator in $\mathcal{A}\cap\mathcal{X}$. Then, by Lemma \ref{lem:(5.1homologia_relativa)} (b),
$\omega$ satisfies the desired equalities of (a). Let us show the statements of
(a1) and (a2). 
\

(a1) By Corollary \ref{cor:(4.9homologia_relativa)}, we have 
\begin{alignat*}{1}
\idr[\mathcal{A}\cap\mathcal{X}][\mathcal{X}] & =\coresdimr{\mathcal{B}}{\mathcal{X}}{\mathcal{X}}=\coresdimr{\mathcal{B}\cap\mathcal{X}}{\mathcal{X}}{\,}=\coresdimr{\mathcal{B}\cap\mathcal{X}}{\mathcal{X}}{\mathcal{X}}\mbox{ and }\\
\idr[\mathcal{A}\cap\mathcal{X}][\mathcal{A}'] & =\coresdimr{\mathcal{B}}{\mathcal{A}'}{\mathcal{X}}=\coresdimr{\mathcal{B}\cap\mathcal{X}}{\mathcal{A}'}{\,}=\coresdimr{\mathcal{B}\cap\mathcal{X}}{\mathcal{A}'}{\mathcal{X}},
\end{alignat*}
where $\mathcal{A}':=\mathcal{A}\cap\mathcal{X}$. On the other hand,
by Lemma \ref{thm:(4.8homologia_relativa)} (c) and Lemma \ref{lem:(4.1homologia_relativa)},
\[
\idr[\mathcal{A}\cap\mathcal{X}][\mathcal{X}]=\pdr[\mathcal{X}][\mathcal{A}\cap\mathcal{X}]=\pdr[\mathcal{A}\cap\mathcal{X}][\mathcal{A}\cap\mathcal{X}]=\idr[\mathcal{A}\cap\mathcal{X}][\mathcal{A}\cap\mathcal{X}].
\]
Now, since $\omega\subseteq\mathcal{B}\cap\mathcal{X}$, we have 
\[
\coresdimx{\mathcal{B}\cap\mathcal{X}}{\mathcal{\mathcal{A}\cap X}}\leq\coresdimx{\omega}{\mathcal{\mathcal{A}\cap X}}\mbox{.}
\]
We claim that $\coresdimx{\omega}{\mathcal{\mathcal{A}\cap X}}\leq\idr[\mathcal{A}\cap\mathcal{X}][\mathcal{A}\cap\mathcal{X}]$.
To show it, we can assume $\idr[\mathcal{A}\cap\mathcal{X}][\mathcal{A}\cap\mathcal{X}]<\infty$.
By Proposition \ref{pro:(5.7homologia_relativa)} (a), we can apply 
Lemma \ref{lem:(5.3homologia_relativa)}
to the pair $(\mathcal{A}\cap\mathcal{X},\omega)$ and since $\idr[\mathcal{A}\cap\mathcal{X}][\mathcal{A}\cap\mathcal{X}]<\infty$,
it follows that 
\[
\mathcal{A}\cap\mathcal{X}=\left\{ Z\in\mathcal{A}\cap\mathcal{X}\,|\:\idr[\mathcal{A}\cap\mathcal{X}][Z]<\infty\right\} =\mathcal{A}\cap\mathcal{X}\cap\omega^{\vee}\subseteq\omega^{\vee}\mbox{ }
\]
and  $\idr[\mathcal{A}\cap\mathcal{X}][\mathcal{A}\cap\mathcal{X}]=\coresdimr{\omega}{\mathcal{A}\cap\mathcal{X}}{\,}\mbox{.}$
Hence
\[
\coresdimr{\mathcal{B}\cap\mathcal{X}}{\mathcal{A}'}{\,}\leq\coresdimr{\omega}{\mathcal{A}'}{\,}\leq\idr[\mathcal{A}\cap\mathcal{X}][\mathcal{A}']=\coresdimr{\mathcal{B}\cap\mathcal{X}}{\mathcal{A}'}{\,}
\]
and thus $\coresdimr{\mathcal{B}\cap\mathcal{X}}{\mathcal{A}\cap\mathcal{X}}{\,}=\coresdimr{\omega}{\mathcal{A}\cap\mathcal{X}}{\,};$ proving (a1).
\

(a2) $(\Rightarrow)$ Let $\pdr[\mathcal{A}\cap\mathcal{X}][\mathcal{A}\cap\mathcal{X}]<\infty.$ Then, by (a1)
$\coresdimr{\mathcal{B}}{\mathcal{X}}{\mathcal{X}}<\infty$
 and thus $\mathcal{X}\subseteq(\mathcal{X},\mathcal{B}){}^{\vee}\subseteq(\mathcal{B}\cap\mathcal{X})^{\vee}$
since $\mathcal{X}$ is closed under extensions. Moreover, in the proof
of (a1), we showed that $\mathcal{A}\cap\mathcal{X}\subseteq\omega^{\vee}$.
Then, by (a1) and Proposition \ref{pro:(5.7homologia_relativa)} (e), we have 
$\pdr[\mathcal{X}][\omega]=\idr[\omega][\mathcal{X}]=\idr[\mathcal{A}\cap\mathcal{X}][\mathcal{X}]=\pdr[\mathcal{X}][\mathcal{A}\cap\mathcal{X}]\mbox{.}$
\

$(\Leftarrow)$ Let $\mathcal{A}\cap\mathcal{X}\subseteq\omega^{\vee}$
and $\pdr[\mathcal{X}][\omega]=n<\infty$. Since $\omega\subseteq\mathcal{A}\cap\mathcal{X}\subseteq\omega^{\vee}$,
by (a1) and Lemma \ref{lem:(4.10homologia_relativa)}, we have $\pdr[\mathcal{X}][\omega]=\pdr[\mathcal{X}][\omega^{\vee}]\geq\pdr[\mathcal{X}][\mathcal{A}\cap\mathcal{X}]=\pdr[\mathcal{A}\cap\mathcal{X}][\mathcal{A}\cap\mathcal{X}].$
\end{proof}

\section{The class $\protect\Gennr[M][n][\mathcal{X}]$}

In the theory of infinitely generated tilting modules of finite projective
dimension, 
S. Bazzoni
 \cite{Bazzonintilting} and 
 J. Wei
  \cite{Wei}
 presented the class $\operatorname{Gen}_{n}(M)$ as a tool in the characterization of tilting modules. The goal of this section
is to present a generalization of such class, and to review some basic
properties, which will be used in  \cite{Argudin-Mendoza2} to characterize when a class $\mathcal{T}\subseteq\C$ is $n$-$\X$-tilting. 

\begin{defn}\label{def:gennr} Let $\mathcal{C}$ be an abelian category, $n\geq1$
and $\mathcal{X},\mathcal{T}\subseteq\mathcal{C}$. We denote by \textbf{$\Gennr$}
the class of all the objects $C\in\mathcal{C}$ admitting an exact sequence
\[
0\rightarrow K\rightarrow T_{n}\xrightarrow{f_{n}}T_{n-1}\to\cdots\to T_2\xrightarrow{f_{2}}T_{1}\xrightarrow{f_{1}}C\rightarrow 0\mbox{,}
\]
with $\Kerx[f_{i}]\in\mathcal{X}$ and $T_{i}\in\mathcal{T}\cap\mathcal{X}$
$\forall\, i\in[1,n]$.

We also define $\operatorname{Gen}_{n}^{\mathcal{X}}(\mathcal{T}):=\Gennr[\mathcal{T}^{\oplus}][n][\mathcal{X}]$
and $\operatorname{gen}_{n}^{\mathcal{X}}(\mathcal{T}):=\Gennr[\mathcal{T}^{\oplus_{<\infty}}][n][\mathcal{X}]$.
For an object $T\in\mathcal{C}$, we define $\operatorname{Gen}_{n}^{\mathcal{X}}(T):=\operatorname{Gen}_{n}^{\mathcal{X}}(\operatorname{Add}(T))$
and $\operatorname{gen}_{n}^{\mathcal{X}}(T):=\operatorname{gen}_{n}^{\mathcal{X}}(\operatorname{add}(T))$.
If $\mathcal{X}=\mathcal{C}$, we set $\Genn[\mathcal{T}]:=\Gennr[\mathcal{T}][][\mathcal{C}]$,
$\operatorname{Gen}_{n}(\mathcal{T}):=\operatorname{Gen}_{n}^{\mathcal{C}}(\mathcal{T})$
and $\operatorname{gen}_{n}(\mathcal{T}):=\operatorname{gen}_{n}^{\mathcal{C}}(\mathcal{T})$.
\end{defn}

The following result is a generalization of \cite[Prop. 3.7]{Wei}.

\begin{pro}\label{pro:Genn cerrado por n-cocientes} Let $\mathcal{C}$ be an
abelian category, $\mathcal{T}\subseteq\mathcal{C},$ $\mathcal{X}=\smdx[\mathcal{X}]\subseteq\mathcal{C}$
be closed under extensions, $\Gennr\cap\mathcal{X}$ be closed
under extensions and let
$\Gennr\cap\mathcal{X}=\Gennr[\mathcal{T}][n+1]\cap\mathcal{X}.$ Then
\begin{center}
 \textup{$\mbox{Fac}_{k}^{\mathcal{X}}(\Gennr)\cap\mathcal{X}=\Gennr[][k]\cap\mathcal{X}\mbox{ }\forall k\geq1$}
 \end{center}
and $\Gennr\cap\mathcal{X}$ is closed by $n$-quotients
in $\mathcal{X}$. 
\end{pro}

\begin{proof}
Since $\mathcal{T}\cap\mathcal{X}\subseteq\Gennr\cap\mathcal{X}$,
it follows that 
\[
\mbox{Fac}_{k}^{\mathcal{X}}(\Gennr)\cap\mathcal{X}\supseteq\Gennr[][k]\cap\mathcal{X}\mbox{.}
\]
 Hence, we need to show that $\mbox{Fac}_{k}^{\mathcal{X}}(\Gennr)\cap\mathcal{X}\subseteq\Gennr[][k]\cap\mathcal{X}$.
We proceed by induction on $k$.

Let $k=1$ and $X\in \mbox{Fac}_{k}^{\mathcal{X}}(\Gennr)\cap\mathcal{X}.$ Then, there is an exact sequence 
\[
\suc[K][M'][X][f]\mbox{ with }X,K\in\mathcal{X}\mbox{ and }M'\in\Gennr\cap\mathcal{X}.
\]
Moreover, by using that $\Gennr\cap\mathcal{X}\subseteq\Gennr[][1]\cap\mathcal{X},$ there is an exact sequence 
\[
\suc[K_{1}][M_{1}][M'][][g]\mbox{ with }M_{1}\in\mathcal{T}\cap\mathcal{X}\mbox{ and }K_{1}\in\mathcal{X}\mbox{.}
\]
By taking the pull-back of $f$ and $g$ and since $\mathcal{X}$
is closed under extensions, we get that $X\in\Gennr[\mathcal{T}][1][\mathcal{X}]\cap\mathcal{X}$.
Therefore $\mbox{Fac}_{1}^{\mathcal{X}}(\Gennr)\cap\mathcal{X}\subseteq\Gennr[][1]\cap\mathcal{X}.$
\

Let $k>1$ and $M\in \mbox{Fac}_{k+1}^{\mathcal{X}}(\Gennr)\cap\mathcal{X}.$ Then, there is an exact sequence 
\[
0\rightarrow K\rightarrow C_{k+1}\overset{f_{k+1}}{\rightarrow}C_{k}\rightarrow...\rightarrow C_{1}\overset{f_{1}}{\rightarrow}M\rightarrow 0,
\]
 with $M\in\mathcal{X}$, $C_{i}\in\Gennr[\mathcal{T}][n]\cap\mathcal{X}$
and $\Kerx[f_{i}]\in\mathcal{X}$ $\forall\, i\in[1,k+1]$. Observe that
\[
M_{1}:=\Kerx[f_{1}]\in\mbox{Fac}_{k}^{\mathcal{X}}(\Gennr)\cap\mathcal{X}\mbox{.}
\]
 By inductive hypothesis $M_{1}\in\Gennr[][k]\cap\mathcal{X},$ and thus, there  is an exact sequence 
\begin{center} 
 $0\to M_{1}\xrightarrow{i}C_{1}\xrightarrow{f_1}M\to 0,$ with $M_{1}\in\Gennr[][k]\cap\mathcal{X}$  and $C_{1}\in\Gennr\cap\mathcal{X}.$
 \end{center}

On the other hand, since $\Gennr\cap\mathcal{X}=\Gennr[][n+1]\cap\mathcal{X}$
and $C_{1}\in\Gennr\cap\mathcal{X}$, \\
\begin{minipage}[t]{0.55\columnwidth}%
there is an exact sequence 
\[
0\rightarrow C'\rightarrow T_{1}\overset{p}{\rightarrow}C_{1}\rightarrow 0,
\]
 with $T_{1}\in\mathcal{T}\cap\mathcal{X}$ and $C'\in\Gennr[\mathcal{T}][n][\mathcal{X}]\cap\mathcal{X}$.
Now, by taking the pull-back of $i$ and $p$, we get an exact sequence
\[
\suc[C'][Y][M_{1}][\,][p']\mbox{;}
\]
and since $M_{1}\in\Gennr[][k]$, we have an exact sequence 
\[
\suc[M'_{1}][T'_{1}][M_{1}][\,][p''],
\]
\end{minipage}\hfill{}%
\fbox{\begin{minipage}[t]{0.4\columnwidth}%
\[
\begin{tikzpicture}[-,>=to,shorten >=1pt,auto,node distance=.9cm,main node/.style=,x=1.5cm,y=1.5cm]

   \node[main node] (M1) at (0,0)      {$M_1$};
   \node[main node] (C1) [right of=M1]  {$C_1$};
   \node[main node] (MA) [right of=C1]  {$M$};

   \node[main node] (Y) [above of=M1]  {$Y$};
   \node[main node] (T1) [right of=Y]  {$T_1$};
   \node[main node] (MB) [right of=T1]  {$M$};

   \node[main node] (CA) [above of=Y]  {$C'$};
   \node[main node] (CB) [above of=T1]  {$C'$};

   \node[main node] (01) [above of=CA]    {$0$};
   \node[main node] (02) [above of=CB]    {$0$};

   \node[main node] (03) [left of=Y]    {$0$};
   \node[main node] (04) [right of=MB]    {$0$};

   \node[main node] (05) [left of=M1]    {$0$};
   \node[main node] (06) [right of=MA]    {$0$};

   \node[main node] (07) [below of=M1]    {$0$};
   \node[main node] (08) [below of=C1]    {$0$};

\draw[->, thin]   (M1)  to node  {$$}  (C1);
\draw[->, thin]   (C1)  to node  {$$}  (MA);
\draw[->, thin]   (05)  to node  {$$}  (M1);
\draw[->, thin]   (MA)  to node  {$$}  (06);

\draw[->, thin]   (Y)  to node  {$$}  (T1);
\draw[->, thin]   (T1)  to node  {$$}  (MB);
\draw[->, thin]   (03)  to node  {$$}  (Y);
\draw[->, thin]   (MB)  to node  {$$}  (04);

\draw[->, thin]   (01)  to node  {$$}  (CA);
\draw[->, thin]   (CA)  to  node  {$$}  (Y);
\draw[->, thin]   (Y)  to  node   {$$}  (M1);
\draw[->, thin]   (M1)  to  node   {$$}  (07);

\draw[-, double]   (CA)  to  node  {$$}  (CB);

\draw[->, thin]   (02)  to  node   {$$}  (CB);
\draw[->, thin]   (CB)  to  node   {$$}  (T1);
\draw[->, thin]   (T1)  to  node   {$$}  (C1);
\draw[->, thin]   (C1)  to  node   {$$}  (08);

\draw[-, double]   (MA)  to  node   {$$}  (MB);
   
\end{tikzpicture}
\]%
\end{minipage}}\\
\vspace{0.2cm}

\noindent \makebox[\linewidth]{with $T'_{1}\in\mathcal{T}\cap\mathcal{X}$ and $M'_{1}\in\Gennr[][k-1]\cap\mathcal{X}$.
By taking the pull-back of $p'$ and $p''$,}
\\
\fbox{\begin{minipage}[t]{0.4\columnwidth}%
\[
\begin{tikzpicture}[-,>=to,shorten >=1pt,auto,node distance=.9cm,main node/.style=,x=1.5cm,y=1.5cm]

   \node[main node] (M1) at (0,0)      {$M_1$};
   \node[main node] (Y) [left of=M1]  {$Y$};
   \node[main node] (CA) [left of=Y]  {$C'$};

   \node[main node] (T1) [above of=M1]  {$T'_1$};
   \node[main node] (X) [above of=Y]  {$X$};
   \node[main node] (CB) [left of=X]  {$C'$};

   \node[main node] (MB) [above of=T1]  {$M'_1$};
   \node[main node] (MA) [above of=X]  {$M'_1$};

   \node[main node] (01) [above of=MA]    {$0$};
   \node[main node] (02) [above of=MB]    {$0$};

   \node[main node] (03) [left of=CB]    {$0$};
   \node[main node] (04) [right of=T1]    {$0$};

   \node[main node] (05) [left of=CA]    {$0$};
   \node[main node] (06) [right of=M1]    {$0$};

   \node[main node] (07) [below of=Y]    {$0$};
   \node[main node] (08) [below of=M1]    {$0$};

\draw[->, thin]   (CA)  to node  {$$}  (Y);
\draw[->, thin]   (Y)  to node  {$$}  (M1);
\draw[->, thin]   (M1)  to node  {$$}  (06);
\draw[->, thin]   (05)  to node  {$$}  (CA);

\draw[->, thin]   (CB)  to node  {$$}  (X);
\draw[->, thin]   (X)  to node  {$$}  (T1);
\draw[->, thin]   (T1)  to node  {$$}  (04);
\draw[->, thin]   (03)  to node  {$$}  (CB);

\draw[->, thin]   (01)  to node  {$$}  (MA);
\draw[->, thin]   (MA)  to  node  {$$}  (X);
\draw[->, thin]   (X)  to  node   {$$}  (Y);
\draw[->, thin]   (Y)  to  node   {$$}  (07);

\draw[-, double]   (MA)  to  node  {$$}  (MB);

\draw[->, thin]   (02)  to  node   {$$}  (MB);
\draw[->, thin]   (MB)  to  node   {$$}  (T1);
\draw[->, thin]   (T1)  to  node   {$$}  (M1);
\draw[->, thin]   (M1)  to  node   {$$}  (08);

\draw[-, double]   (CA)  to  node   {$$}  (CB);
   
\end{tikzpicture}
\]%
\end{minipage}}\hfill{}%
\begin{minipage}[t]{0.55\columnwidth}%
we get an exact sequence
\[
\suc[C'][X][T'_{1}][\,][\,],
\]
with $C'\in\Gennr\cap\mathcal{X}$ and 
\[
T'_{1}\in\mathcal{T}\cap\mathcal{X}\subseteq\Gennr\cap\mathcal{X}\mbox{,}
\]
 and thus, $X\in\Gennr\cap\mathcal{X}$. On the other hand, from the
second pull-back we also get an exact sequence 
\[
\suc[M'_{1}][X][Y][\,][\,]\mbox{,}
\]
with   $M'_{1}\in\Gennr[\mathcal{T}][k-1]\cap\mathcal{X}\subseteq\mbox{Fac }_{k-1}^{\mathcal{X}}\left(\Gennr\right)\cap\mathcal{X}$ and  $X\in\Gennr\cap\mathcal{X}$.
\end{minipage}

\noindent  Hence, $Y\in\mbox{Fac }_{k}^{\mathcal{X}}\left(\Gennr\right)\cap\mathcal{X}\subseteq\Gennr[][k]\cap\mathcal{X}$
by inductive hypothesis. Finally, from the first pull-back, we get the exact sequence 
\[
\suc[Y][T_{1}][M][\,][\,]\mbox{,}
\]
where $Y\in\Gennr[][k]\cap\mathcal{X}$ and $T_{1}\in\mathcal{T}\cap\mathcal{X}$.
Therefore $
M\in\Gennr[][k+1]\cap\mathcal{X};$ proving the result. 

Lastly,  the fact that   $\Gennr[][n]\cap\mathcal{X}$ is closed by $n$-quotients follows from the equality $\mbox{Fac}_{n}^{\mathcal{X}}(\Gennr)\cap\mathcal{X}=\Gennr[][n]\cap\mathcal{X}$.
\end{proof}

\bigskip
{\bf Acknowledgements} We thank an anonymous referee whose comments contributed to improve the quality of the paper.

\bibliographystyle{plain}
\bibliography{nXtiltingV9}




\end{document}